%% file: 00_peri_fmo.tex
\documentclass[a4paper,11pt]{article}
\usepackage{fmo-modern}
\newcommand{\R}{\mathbb R}\newcommand{\Rn}{\mathbb R^d}
\newcommand{\Sn}{\mathbb S^{d-1}}
\renewcommand{\O}{\Omega}
\newcommand{\Od}{\Omega_\delta}
\newcommand{\dx}{\,\mathrm dx}

\newcommand{\sym}{\mathrm s}
\newcommand{\asym}{\mathrm a}
\newcommand{\kadm}{\mathcal K}
\newcommand{\wsto}{\overset{*}{\rightharpoonup}}

\newcommand{\mres}{\mathbin{\vrule height 1.6ex depth 0pt width
0.13ex\vrule height 0.13ex depth 0pt width 1.3ex}}

\usepackage{enumitem}

\newtheorem{theorem}{Theorem}[section]\newtheorem{lemma}[theorem]{Lemma}\newtheorem{proposition}[theorem]{Proposition}\newtheorem{corollary}[theorem]{Corollary}
\theoremstyle{remark}\newtheorem{remark}[theorem]{Remark}\theoremstyle{definition}
\DeclareMathOperator{\Div}{div}\DeclareMathOperator{\supp}{spt}\DeclareMathOperator{\closure}{cl}
\providecommand{\doi}[1]{\href{https://doi.org/#1}{doi:#1}}
\title{Free-Material Optimization in a Peridynamic World}
\author{\authororcid{Anton Evgrafov}{0000-0002-3987-7745}\\Department of Mathematical Sciences, Aalborg University\\\texttt{anev@math.aau.dk}}
\date{October 14, 2026}
\hypersetup{pdftitle={Free-Material Optimization in a Peridynamic World},pdfauthor={Anton Evgrafov}}
\begin{document}\maketitle
\begin{abstract}
We study free-material optimization for nonlinear bond-based nonlocal diffusion and its localization as the interaction horizon \(\delta\) vanishes. Bond conductivities are subject to pointwise bounds and a global material budget constraint. After normalization, the asymptotic behavior of these problems is governed by the fraction \(\theta_\delta\) of potentially active bonds that can be assigned the maximal conductivity. 

We identify three distinct limiting regimes. 
If \(\theta_\delta\to0\), the normalized optimal values converge to those of the classical vanishing-material optimal design problem, equivalently formulated as a Lipschitz-constrained state problem or a measure-valued flux problem. 
For \(\theta_\delta\to\theta\in(0,1)\), we obtain a directional local free-material problem governed by the upper tail of the directional \(p\)-strain over position-direction space. 
Thus the state variable localizes, while the limiting constitutive law retains an internal directional variable inherited from the bonds. 
If \(\theta_\delta\to\theta\geq1\), the limit is the isotropic local \(p\)-Dirichlet problem. 

A convex minimax argument eliminates the normalized conductivity variable exactly, after which the analysis combines convex duality, nonlocal compactness, and \(\Gamma\)-convergence for convex integrands.
We prove a unified \(\Gamma\)-convergence result for the reduced state energies whenever \(\theta_\delta\to\theta\in[0,+\infty)\), yielding convergence of normalized optimal values and compactness and identification of optimal states throughout the full parameter range. At the vanishing-material endpoint, we additionally recover optimal measure-valued local fluxes from optimal nonlocal two-point fluxes.
\end{abstract}


\input{01_introduction.tex}

\input{02_setup.tex}
\input{03_gamma.tex}
\input{04_flux_recovery.tex}
\input{05_discussion.tex}

\bibliographystyle{alpha}
\bibliography{nloc_frimatopt}

\appendix
\input{appx_A.tex}
\input{appx_B.tex}
\input{appx_C.tex}

\end{document}

%% file: 01_introduction.tex
\section{Introduction}\label{sec:intro}
Free-material optimization (FMO) is an optimal design paradigm in which the material constitutive tensor is treated as a pointwise design variable; see the original papers~\cite{kovcvara1998free,kovcvara2008free} as well as their discussion in monographs such as~\cite{cherkaev2012variational,allaire2012shape,bendsoe2013topology}. Although FMO may be regarded as a ``thought experiment,'' it provides rigorous bounds on what is achievable when the designer is constrained only by the fundamental physical requirements on the constitutive tensor, such as definiteness and symmetry, but not by homogeneity or isotropy.

In this work our starting point is an FMO problem within the bond-based peridynamic 
setting. 
The constitutive variable is thus attached not to a spatial point, but to a bond joining two interacting points.
We use a nonlinear nonlocal diffusion model of \(p\)-Laplacian type, with \(1<p<+\infty\), as our model. 
In schematic form, the potential formulation is
\begin{equation*}
    \sup_{\kappa }
    \inf_{\substack{u \in L^{p}(\Rn)\\u|_{\Rn\setminus\O} = 0}}
    \left\{
        \frac{1}{p}\iint_{|x-x'|<\delta} \kappa(x,x')\frac{|u(x)-u(x')|^{p}}{|x-x'|^p}
        \,\dx\dx' - \int_{\O} f(x)u(x)\dx
    \right\},
\end{equation*}
where \(\delta>0\) is the interaction horizon, \(\kappa(x,x')\) is the symmetric bond conductivity, \(f\) is the source, and the state vanishes outside the design domain \(\O\).\footnote{%
Equivalently, the problem can be formulated as minimization of the complementary energy over the conductivity and a two-point flux satisfying a nonlocal divergence constraint. 
The precise nonlocal operators, admissible material classes, normalizations, and optimal values are introduced in Section~\ref{sec:setup}.}
Such an optimization problem answers the question: \emph{if one is free to choose the properties of each bond in a bond-based peridynamic material, how should this choice be made in an optimal fashion?} 
This formulation aligns directly with the spirit of FMO, in which the designer is granted maximal freedom subject only to the intrinsic physical constraints of the underlying model.

In this work, we take this line of thought one step further.
Namely, we investigate \emph{what local optimal design problem emerges from the bond-based peridynamic FMO problem as the nonlocal horizon is driven to zero.}
There are several related studies addressing similar questions.
We mention, for example, \cite{AnMuRo19,AnMuRo21,andres2015type,andres2023minimization,mengesha2025asymptotic,evgrafov2019non,evgrafov2021nonlocal,evgrafov2024nonlocal,schytt2024dual}.
The main distinction of the present work is that we impose no a priori relation between 
the nonlocal two-point bond conductivity and a local one-point conductivity, such as an 
averaging relation.

 For the purposes of this introduction, it is sufficient to note that the asymptotic behavior is governed by a dimensionless relative material fraction, denoted here by 
 \(\theta_\delta\); see~\eqref{eq:relative-material-fraction} for a precise definition.
 This parameter compares the available conductivity budget with the amount required to assign the maximal admissible conductivity to every potentially active bond.
 Thus \(0<\theta_\delta<1\) means that only a fraction of the active bond set can be saturated, while \(\theta_\delta\geq1\) means that there is enough material to saturate all active bonds. 
 Our results reveal the following emergent three-regime phase diagram for the appropriately normalized problems:
 \[ \begin{array}{ccl} 
  \theta_\delta\to 0 &\Longrightarrow& \text{the vanishing-material optimal design problem}, \\[1mm] 
  \theta_\delta\to\theta\in(0,1) &\Longrightarrow& \text{a directional local free-material problem}, \\[1mm] 
  \theta_\delta\to\theta\in [1,+\infty) &\Longrightarrow& \text{the isotropic \(p\)-Dirichlet problem}. 
\end{array} \]
The two endpoint regimes recover classical local variational problems. 
In the vanishing-material regime, the limit is the optimal design problem, which admits equivalent formulations as a Lipschitz-constrained state problem or an \(L^1\)-flux problem and its measure-valued relaxation; see~\cite{evgrafov2021nonlocal} for a closely related work.
At the opposite endpoint, corresponding to \(\theta\geq1\), the material constraint becomes asymptotically inactive and the limit is the standard isotropic \(p\)-Dirichlet problem.

The intermediate regime is qualitatively different. 
The state variable localizes, but the limiting constitutive law retains an internal directional variable inherited from the bonds. 
Its energy depends on the distribution of the directional strains 
\[ |\nabla u(x)\cdot s|^p, \qquad (x,s)\in\O\times\mathbb{S}^{d-1}. \] 
Consequently, the state problem localizes, while the constitutive description retains a directional internal variable.
In concise terms, localization eliminates the bond length while retaining the bond direction. 
The resulting directional free-material problem is defined precisely in Section~\ref{sec:intermediate-material} and interpreted further in Section~\ref{sec:discussion}.

The three regimes are covered by a single variational convergence result. After exact elimination of the normalized conductivity variable, the normalized design problems admit reduced state energies \(\mathcal{E}_{\delta,\theta_\delta}\). If \[ \theta_\delta\to\theta\in[0,+\infty), \] then \[ \mathcal{E}_{\delta,\theta_\delta} \xrightarrow{\Gamma} \mathcal{E}_\theta \qquad\text{in }L^p(\O), \] and the family on the left is equicoercive. Consequently, the normalized optimal values converge, every sequence of optimal reduced states is relatively compact in \(L^p(\O)\), and every cluster point is an optimal state of the limiting problem. The limiting energy \(\mathcal{E}_\theta\) is the Lipschitz energy when \(\theta=0\), the directional upper-tail energy when \(0<\theta<1\), and the rescaled \(p\)-Dirichlet energy when \(\theta\geq1\). The precise statements are given in Proposition~\ref{prop:intermediate-gamma-limit} and Theorem~\ref{thm:intermediate-localization}.

This reduction is obtained by a convex minimax argument and produces an upper-tail functional of the nonlocal directional strain. A scalar dual representation reduces the localization argument to \(\Gamma\)-convergence for fixed convex nonlocal integrands, allowing the sparse, intermediate, and saturated regimes to be treated within one proof. 

At the sparse endpoint, we obtain an additional optimizer-level result. Optimal nonlocal two-point fluxes have asymptotically minimal normalized \(L^1\)-norm, and their recovered one-point fluxes are weak-star precompact as vector-valued Radon measures. Every weak-star cluster point of the recovered flux measures solves the local measure-valued  optimal design problem. We also record an exact reference result: after removing the pointwise conductivity cap, the reduced nonlocal value is independent of the interaction horizon.

The paper is organized as follows. Section~\ref{sec:setup} introduces the nonlocal state and flux formulations, the bond-based free-material problem, and its normalization. Section~\ref{sec:intermediate-material}  eliminates the normalized conductivity variable and proves the unified localization result for \(\theta_\delta\to\theta\in[0,+\infty)\), including the sparse and saturated endpoint identifications.
In Subsection~\ref{subsec:p-one}, we also record that the reduced state-energy localization extends to \(p=1\), with a BV-type limiting energy.
Section~\ref{sec:sparse-flux-recovery} establishes measure-valued flux recovery at the sparse endpoint. Section~\ref{sec:discussion} interprets the directional constitutive laws arising in the intermediate regime and compares them with local free-material optimization and classical composite bounds. Appendix~\ref{app:two-point} collects elementary facts about symmetric and antisymmetric two-point fields. Appendix~\ref{appx:loc_rlx} establishes the local sparse optimal design duality and its measure relaxation, while Appendix~\ref{app:uncapped-exact-localization} studies the conductivity-uncapped reference problem and proves its exact horizon-independence.

%% file: 02_setup.tex
\section{Problem setup}\label{sec:setup}

\subsection{Notation and conventions} 

Throughout the paper, \(\O\subset\Rn\), \(d\geq2\), is a connected bounded Lipschitz domain. We fix \(1<p<+\infty\) and denote its H{\"o}lder conjugate by \( q=p/(p-1). \) 
The norm of a finite-dimensional vector is denoted by \(|\cdot|\), while \(\|\cdot\|_X\) denotes the norm in a function space \(X\). 
We also write \(|A|\) for the Lebesgue measure of a measurable set \(A\). 
For \(r>0\), we denote by \(B_r(x)\) the open ball of radius \(r\) centred at \(x\) in the ambient Euclidean space, and write \(B_r =B_r(0)\). 
The ambient dimension will always be clear from the context.
We denote the unit sphere in \(\Rn\) by \( \mathbb{S}^{d-1}=\{s\in\Rn: |s|=1\}. \)
For a function or measure \(\phi\), we denote its support by \(\operatorname{spt}\phi\).
For a set \(A\subset\R^n\), we write \(\operatorname{cl}A\) for its closure.
The characteristic function of a measurable set \(A\) is denoted by \(\mathbf{1}_A\). 
For a subset \(A\) of a linear space, its convex indicator is 
\[ \iota_A(x) = \begin{cases} 0, & x\in A,\\ +\infty, & x\notin A. \end{cases} \] 
For \(t\in\R\), we set \( t_+=\max\{t,0\}. \)
We use the convention \(\overline{\R} = \R\cup \{+\infty\}\). 
We use \(\mathcal{L}^d\) for \(d\)-dimensional Lebesgue measure, \(\mathcal{H}^{d-1}\) 
for \((d-1)\)-dimensional Hausdorff measure, and \( \mathcal{M}(K;\Rn) \) for the space of finite \(\Rn\)-valued Radon measures on a compact set \(K\), equipped with the total-variation norm.
Integration with respect to \(\mathcal{L}^d\) is denoted by \(\mathrm{d}x\), and 
\(v\mathcal{L}^d\) denotes the measure with density \(v\) with respect to Lebesgue measure. 
We write \( \mathcal{D}(\O)=C_c^\infty(\O) \) and denote its distributional dual by \(\mathcal{D}'(\O)\).
The notation \( \langle\cdot,\cdot\rangle_\O \) is used for the relevant duality pairing
over \(\O\). 
When both arguments are integrable functions, it agrees with the usual integral pairing.
Every function defined on \(\O\) and prescribed to vanish outside \(\O\) is identified with its zero extension to \(\Rn\). 
In particular, \(\operatorname{Lip}_0(\O)\) denotes the space of globally Lipschitz functions on \(\Rn\) that vanish on \(\Rn\setminus\O\), and \(\operatorname{Lip}(u)\) denotes the global Lipschitz constant of \(u\). 
Unless stated otherwise, equalities and inequalities between measurable functions are understood almost everywhere. 
Strong, weak, and weak-star convergence are denoted respectively by 
\(\to\), \(\rightharpoonup\), \(\overset{*}{\rightharpoonup}.\)
For families of functionals, \(\xrightarrow{\Gamma}\) denotes \(\Gamma\)-convergence with respect to the topology specified in the corresponding statement.

\subsection{Bond-based peridynamic operators}\label{subsec:bb_ops}

For a given \emph{nonlocal horizon} \(\delta > 0\) let 
\(\Od = \O + B_\delta\) be the set of points in \(\Rn\), which can interact with points in \(\O\).
We also denote by 
\[ E_\delta = \left\{ (x,x')\in\Od^2: |x-x'|<\delta \right\}, \] 
the set of potentially active bonds.
Let the nonlocal boundary be \(\Gamma_\delta = \Od\setminus \O\).
The intensity of the nonlocal interactions as a function of the distance between two points will be described by a radial nonnegative function \(\rho_\delta: \Rn \to [0,+\infty)\).
We will assume that \(\supp \rho_\delta \subset B_\delta\), and the fact that \(\rho_\delta\) satisfies the normalization condition
\begin{equation}\label{eq:rho0}
    \int_{B_\delta} \rho_\delta(x)\,\dx = K_{d,p}^{-1}, \quad\text{where}\quad
    K_{d,p} = |\Sn|^{-1}\int_{\Sn} |e\cdot s|^p\,\mathrm{d}s
    =\frac{\Gamma(\frac{d}{2})\Gamma(\frac{p+1}{2})}{\pi^{1/2}\Gamma(\frac{d+p}{2})},
\end{equation}
and \(e\in\mathbb{S}^{d-1}\) is arbitrary.
For a function \(u:\O\to\R\) we define the \emph{nonlocal gradient} by
\begin{equation*}
    D^\rho_\delta u (x,x') = \frac{u_0(x)-u_0(x')}{|x-x'|}\rho_{\delta}^{1/p}(x-x'),
\end{equation*}
where \(u_0\) is the extension of \(u\) by \(0\) outside of \(\O\).
Note that \(D^\rho_\delta u (x,x') = 0\) whenever \(x-x' \not\in B_\delta\), or indeed whenever \(x \not\in \Od\) or \(x' \not\in \Od\).
We shall consider it as a linear, possibly unbounded operator \(D^\rho_\delta: D(D^\rho_\delta) \subset L^p(\O) \to L^p(\Od^2)\).
\(D^\rho_\delta\) shares several properties with the usual distributional gradient.
For example, it is densely defined in \(L^p(\O)\), because \(C^\infty_c(\O)\subset D(D^\rho_\delta)\).
The operator \(D^\rho_\delta\) is also closed. 
Indeed, if \(u_k\to u\) in \(L^p(\O)\) and \(D^\rho_\delta u_k\to v\) in \(L^p(\Od^2)\), then, after passing to a subsequence, both convergences hold almost everywhere, and the pointwise definition of \(D^\rho_\delta\) yields \(v=D^\rho_\delta u\) almost everywhere. Moreover \(D^\rho_\delta\) 
satisfies a nonlocal Poincar{\'e}-type inequality, see~\cite{bourgain2001another,ponce2004new,ponce2004estimate}, 
uniformly for small \(\delta>0\):
\begin{equation*}
    \exists C>0, \delta_0>0: \forall \delta \in (0,\delta_0), u\in D(D^\rho_\delta): \lVert u\rVert_{L^p(\O)} \leq C \lVert D^\rho_\delta u\rVert_{L^p(\Od^2)},
\end{equation*}
which makes \(D^\rho_\delta\) injective.
Its graph norm is therefore equivalent to \(\lVert D^\rho_\delta(\cdot)\rVert_{L^p(\Od^2)}\).
All this allows us to consider a Banach space \(X_\delta^p = (D(D^\rho_\delta), \lVert \cdot\rVert_{X_\delta^p} = \lVert D^\rho_\delta(\cdot)\rVert_{L^p(\Od^2)})\), which may be regarded as the peridynamic analogue of \(W^{1,p}_0(\O)\).

We can now define a \emph{nonlocal divergence} operator 
\[
\begin{aligned}
\operatorname{div}^\rho_\delta: 
D(\operatorname{div}^\rho_\delta) = \bigg\{\,
 &\sigma \in L^q(\Od^2) \colon \exists g \in L^q(\O), \text{such that\ }
 \\&\langle g, u \rangle_\O = -\iint_{\Od^2} \sigma D^\rho_\delta u\, \forall u \in X_\delta^p
\,\bigg\} \to L^q(\O),
\end{aligned}
\] 
as the negative adjoint of \(D^\rho_\delta\).
In other words, we define \(\operatorname{div}^\rho_\delta\) through the following integration by parts formula:
\begin{equation*}
    \langle \operatorname{div}^\rho_\delta\sigma, u\rangle_{\O} = -\iint_{\Od^2} \sigma D^\rho_\delta u\dx'\dx, \qquad\forall (\sigma,u) \in D(\operatorname{div}^\rho_\delta)\times D(D^\rho_\delta).
\end{equation*}
The operator \(\operatorname{div}^\rho_\delta\) is closed (as an adjoint operator);
surjective (owing to the closed-range theorem~\cite[Theorem~2.21]{brezis2010functional} and
the nonlocal Poincar{\'e} inequality); and densely defined.
Indeed, for each sufficiently regular \(\sigma\), for example \(\sigma \in C^\infty_c(\Od^2)\), we have an explicit expression for  \(\operatorname{div}^\rho_\delta \sigma\):
\begin{equation*}
    \begin{aligned}
        &-\iint_{\Od^2} \sigma D^\rho_\delta u\dx'\dx
    =\int_{\O} u(x) \underbrace{\int_{\Od} \frac{\sigma(x',x)-\sigma(x,x')}{|x-x'|}\rho_{\delta}^{1/p}(x-x')\,\mathrm{d}x'}_{=\operatorname{div}^\rho_\delta\sigma(x)}\,\mathrm{d}x,
    \quad \forall u \in C^\infty_c(\O).
    \end{aligned}
\end{equation*}
Note that the formula above is warranted because, for smooth \(\sigma\),
\[|\sigma(x',x)-\sigma(x,x')|\leq \sqrt{2}\lVert\nabla\sigma\rVert_{L^\infty(\Od^2;\R^{2d})}|x-x'|,\]
and consequently
\[
|\operatorname{div}^\rho_\delta\sigma(\cdot)| \leq \sqrt{2}\lVert\nabla\sigma\rVert_{L^\infty(\Od^2;\R^{2d})}
\int_{B_\delta}\rho_{\delta}^{1/p}(z)\,\mathrm{d}z
\in L^q(\O),
\]
since
\[
\int_{B_\delta}\rho_{\delta}^{1/p}(z)\,\mathrm{d}z \leq |B_\delta|^{1/q}
\bigg(\int_{B_\delta}\rho_{\delta}(z)\,\mathrm{d}z\bigg)^{1/p} = |B_\delta|^{1/q}K_{d,p}^{-1/p},
\]
owing to H{\"o}lder's inequality and the normalization of \(\rho_\delta\).

With this in mind we define the Banach space \(Y_\delta^q = (D(\operatorname{div}^\rho_\delta), \lVert\cdot \rVert_{Y_\delta^q} = \lVert \cdot\rVert_{L^q(\Od^2)} + \lVert \operatorname{div}^\rho_\delta(\cdot)\rVert_{L^q(\O)})\), which is a peridynamic analogue of the distributional \(W^q(\Div;\O)\) spaces. 
For a given fixed source \(f\in L^q(\O)\) we will denote by \(Y_\delta^q(f) = \{\, \sigma \in Y_\delta^q \colon \operatorname{div}^\rho_\delta \sigma = f\,\}\) a nonempty closed affine subspace  of \(Y_\delta^q\).

\subsection{Nonlocal state and flux problems for a fixed conductivity}

We next formulate the primal and complementary variational problems for a fixed distribution
of bond conductivities.
These problems provide the state equation and flux formulation underlying the free-material
optimization problem introduced below in Subsection~\ref{subsec:FMO}.

For a measurable set \(S\subset\Rn\times\Rn\), let 
\(S^{\mathrm T}=\{(x,x')\in\Rn\times\Rn: (x',x)\in S\}. \)
For a measurable two-point field \(\phi\), we write \( \phi^{\mathrm T}(x,x')=\phi(x',x). \) 
On a symmetric measurable set \(S=S^{\mathrm T}\), we denote by 
\(L^r_{s}(S)\) and \(L^r_{a}(S)\) the closed subspaces of symmetric and antisymmetric
\(L^r(S)\)-fields, respectively: 
\[ 
  L^r_{\sym}(S) = \{\phi\in L^r(S):\phi^{\mathrm T}=\phi\}, \qquad 
  L^r_{\asym}(S) = \{\phi\in L^r(S):\phi^{\mathrm T}=-\phi\}. 
\]
Further properties of these spaces are collected in Appendix~\ref{app:two-point}.

Let \(0\leq \underline{\kappa} < \overline{\kappa} < \infty\) be the lower and the upper bounds on the conductivities of peridynamic bonds.
We define the set of admissible bond conductivities to be
\begin{equation*}
    \begin{aligned}
    \kadm{(\delta,\underline{\kappa},\overline{\kappa})} &= \{\, \kappa \in L^1(\Od^2)
    \colon \underline{\kappa} \leq \kappa(x,x') \leq \overline{\kappa}, \text{a.e.\ in\ } \Od^2,
    \,\},\\
    \kadm^{\sym}{(\delta,\underline{\kappa},\overline{\kappa})} &= \kadm{(\delta,\underline{\kappa},\overline{\kappa})} \cap L^1_{\rm s}(\Od^2).
    \end{aligned}
\end{equation*}
Although the physical design problem will be posed over symmetric conductivities, 
we retain the larger class \(\kadm{(\delta,\underline{\kappa},\overline{\kappa})}\). 
It will be useful when eliminating the conductivity variable, where symmetry of the 
optimal conductivity will follow from the structure of the optimal flux.
Further, for \(\kappa \in \kadm^{\sym}{(\delta,\underline{\kappa},\overline{\kappa})}\) we define the convex nonlocal potential and complementary energies of such a system:
\begin{equation}\label{eq:energies}
    \begin{aligned}
        \mathcal{P}_\delta(\kappa,u) &= \frac{1}{p} \iint_{\Od^2} \kappa(x,x') |D^\rho_\delta u(x,x')|^p \,\dx'\dx - \langle f,u\rangle_{\O}, \qquad\text{and}\\
        \mathcal{C}_\delta(\kappa,\sigma) &= \frac{1}{q} \iint_{\Od^2} \kappa^{1-q}(x,x') |\sigma(x,x')|^q \,\dx'\dx,\\
    \end{aligned}
\end{equation}
where we understand \(a^{1-q} b^q\) as \(\iota_{\{0\}}(b)\) for \(a=0\).
We now define two  variational principles, related via convex duality, characterizing the nonlinear peridynamic diffusion for a given distribution of bond properties~\cite{schytt2024dual}:
\begin{equation}\label{eq:var_pp}
    \begin{aligned}
    c_\delta(\kappa) = \inf_{\sigma \in Y_\delta^q(f)} \mathcal{C}_{\delta}(\kappa,\sigma) 
    = -\inf_{u \in X_\delta^p} \mathcal{P}_\delta(\kappa,u).
    \end{aligned}
\end{equation}
Since \(|D^\rho_\delta u|^p\) is symmetric, \(\mathcal{P}_\delta(\kappa,u)\) depends only on the symmetric
part \(\kappa_{s} =\frac{1}{2}(\kappa+\kappa^T) \) of \(\kappa\); see 
Lemma~\ref{lem:symmetric-antisymmetric-fields}. 
Therefore, restricting the admissible conductivities to \(L^1_{s}(\Od^2)\) entails no loss
of generality.

The following result establishes well-posedness of the fixed-conductivity problems and
allows optimal fluxes to be chosen antisymmetric.
For details, we refer to~\cite{schytt2024dual}.
\begin{proposition}\label{prop:existence_varpinc}
    Let \(\kappa\in \kadm^{\sym}{(\delta,\underline{\kappa},\overline{\kappa})}\) and \(f\in L^q(\O)\) be arbitrary.
    The following statements about the variational principles~\eqref{eq:var_pp} hold:
    1. The complementary energy attains its infimum at a unique point, which is \emph{antisymmetric}.
    2. The potential energy attains its infimum at a unique point, provided additionally that \(\kappa \geq \underline{\kappa}>0\).
\end{proposition}

\subsection{Free material optimization for bond-based peridynamic diffusion}
\label{subsec:FMO}

Although free-material optimization allows a wide variety of objective functionals and design constraints, we focus here on compliance minimization under a global material constraint. 
Equivalently, we minimize the complementary energy over the bond conductivity and the associated divergence-feasible flux.
For fixed \(\delta>0\), bounds \(0\leq\underline{\kappa}<\overline{\kappa}<+\infty\), 
and \(\gamma\in[\underline\kappa,\overline\kappa]\), we define 
\begin{equation}\label{eq:kappa_adm}
    \begin{aligned}
    \kadm{(\delta,\gamma,\underline{\kappa},\overline{\kappa})} &= 
    \left\{\, \kappa \in \kadm{(\delta,\underline{\kappa},\overline{\kappa})} \colon
    \iint_{\Od^2} \kappa \leq \gamma |\Od|^2\,\right\},
    \qquad\text{and}\\
    \kadm^{\sym}{(\delta,\gamma,\underline{\kappa},\overline{\kappa})} &= \kadm{(\delta,\gamma,\underline{\kappa},\overline{\kappa})} \cap L^1_{\rm s}(\Od^2).
    \end{aligned}
\end{equation}
The parameter \(\gamma\) prescribes the average amount of conductivity available over 
\(\Od^2\).
 Since only the bonds in \(E_\delta\) contribute to the state and complementary energies, 
 it is useful to introduce the dimensionless relative material fraction 
 \begin{equation} \label{eq:relative-material-fraction} 
    \theta_{\delta,\gamma} = \frac{ \gamma|\Od|^2 }{ \overline{\kappa}|E_\delta| }. 
\end{equation}
When \(\underline\kappa=0\) and \(0<\theta_{\delta,\gamma}<1\), this quantity is the
fraction of the potentially active bond set that could be assigned the maximal conductivity 
\(\overline\kappa\). 
If \(\theta_{\delta,\gamma}\geq1\), the available material is sufficient to saturate 
every active bond, and the global material constraint is no longer active.

Problems of this type occur frequently in the structural optimization literature~\cite{allaire2012shape,cherkaev2012variational,bendsoe2013topology}.
They serve both as a simple prototype amenable to theoretical analysis, and an important practical problem in the context of linear elasticity corresponding to finding the stiffest solid given a limited amount of material -- hence the name \emph{compliance minimization}.
The resulting problem can be written as a convex minimization problem with respect to \((\kappa,\sigma)\) or a saddle-point problem with respect to \((\kappa,u)\):
\begin{equation}\label{eq:prob1}
    \begin{aligned}
    i{(\delta,\gamma,\underline{\kappa},\overline{\kappa})}=
    \inf_{\substack{\kappa \in \kadm^{\sym}{(\delta,\gamma,\underline{\kappa},\overline{\kappa})}\\
    \sigma\in Y_\delta^q(f)}} \mathcal{C}_\delta(\kappa,\sigma)
    =
    -\sup_{\kappa \in \kadm^{\sym}{(\delta,\gamma,\underline{\kappa},\overline{\kappa})}} \inf_{u \in X_\delta^p} \mathcal{P}_\delta(\kappa,u).
    \end{aligned}
\end{equation}
We note that the infimum of the complementary energy is attained, see~\cite[Theorem~3.1]{schytt2024dual}.
Moreover, for every optimal conductivity, the associated optimal flux \(\sigma\) is 
antisymmetric, see Proposition~\ref{prop:existence_varpinc} and~\cite{schytt2024dual}.
In the PDE context, a similar saddle-point problem has been studied in~\cite{cea1970example}.

In the asymptotic analysis, both the horizon \(\delta\) and the material budget 
\(\gamma=\gamma_\delta\) may vary. 
The limiting behavior will be organized by the relative material fraction 
\( \theta_\delta = \theta_{\delta,\gamma_\delta}. \) 
No particular relation between \(\gamma_\delta\) and \(\delta\) is imposed at this stage.

\subsection{Standing assumptions and normalized variables}\label{subsec:rescaling}
In the forthcoming asymptotic analysis, we will make the following assumptions
about the nonlocal kernel and the material constants.
\begin{enumerate}[label=(A\arabic*)]
    \item \label{A1}
    The lower conductivity vanishes:
     \(\underline{\kappa}=0\), so that admissible conductivities may vanish on 
    inactive bonds outside of \(E_\delta\);\footnote{We note that a positive lower bound \(\underline{\kappa}_\delta\) converging sufficiently rapidly to zero as \(\delta\downarrow0\) is also possible.}
    \item \label{A2} 
    The interaction kernel is uniform within the nonlocal horizon-ball:
    \(\rho_\delta(x) = \hat{\rho}_\delta \boldsymbol{1}_{B_\delta}(x)\), where \(\hat{\rho}_\delta = [K_{d,p}|B_\delta|]^{-1}\); see~\eqref{eq:rho0}.
\end{enumerate}
Assumption \textnormal{(A2)} aligns the interaction measure with the unweighted bond-material constraint.
In particular, after the normalization below, material occupation is measured with respect to the uniform probability measure on the active bond set. 
For a general nonconstant radial kernel, the material design could also select favorable bond lengths, and an additional radial variable could persist in the limiting problem.

We equip the active bond set \(E_\delta\) with the uniform probability measure 
\[ \mathrm{d}\lambda_\delta(x,x') = \frac{\mathbf{1}_{E_\delta}(x,x')}{|E_\delta|} \dx'\dx. \] 
Unless stated otherwise, all function spaces over \(E_\delta\) below are taken with respect to \(\lambda_\delta\).

We will now simultaneously perform several normalizations.
First, we introduce the normalized state and flux variables:
\begin{equation*}
  \widetilde{u} = (\gamma\hat{\rho}_{\delta})^{q-1}u,
  \qquad\text{and}\qquad
  \widetilde{\sigma} = \widehat{\rho}_\delta^{1/p}|E_\delta|\sigma.
\end{equation*}
Then, we rescale the nonlocal operators as:
\begin{equation*}\label{eq:rescaled_operators}
    \begin{aligned}
        D_\delta
         &= \hat{\rho}_\delta^{-1/p}D^\rho_\delta, &{}\qquad \text{that is} \qquad 
         D^\rho_\delta u &= \hat{\rho}_{\delta}^{1/p} (\gamma\hat{\rho}_{\delta})^{1-q}D_\delta \widetilde{u},\\
    \operatorname{div}_\delta &= \hat{\rho}_\delta^{-1/p}|E_\delta|^{-1}\operatorname{div}^\rho_\delta,  &{}\qquad\text{so that} \qquad \operatorname{div}^\rho_\delta\sigma &= \operatorname{div}_\delta\widetilde{\sigma}.
    \end{aligned}
\end{equation*} 
Since we are simply scaling a linear operator with a positive constant (for each fixed \(\delta\)), this does not change its properties such as the domain, range, closedness, injectivity, or surjectivity. 
Therefore, we retain the symbols \(X_\delta^p\) and \(Y_\delta^q\) for the domains of the normalized operators, now equipped with the norms \( \|u\|_{X_\delta^p} = \|D_\delta u\|_{L^p(E_\delta,\lambda_\delta)} \) and \(\|\sigma\|_{Y_\delta^q} = \|\sigma\|_{L^q(E_\delta,\lambda_\delta)} + \|\operatorname{div}_\delta\sigma\|_{L^q(\O)}. \) 
The normalized state space \(X_\delta^p\) has the same underlying elements as its counterpart in Subsection~\ref{subsec:bb_ops}.
For the flux space, we identify fields that agree on \(E_\delta\), or, equivalently, identify fields on \(E_\delta\) with their zero extensions to \(\Od^2\). 
Under this identification, the normalized spaces have the same underlying elements as their counterparts in Subsection~\ref{subsec:bb_ops}.
As before, we put
\(Y_\delta^q(f) = \left\{ \sigma\in Y_\delta^q: \operatorname{div}_\delta\sigma=f \right\}. \)
Note that the new duality pairing between \(D_\delta\) and \(\operatorname{div}_\delta\) is now expressed using \(\lambda_\delta\):
\[
  \langle \operatorname{div}_\delta \widetilde{\sigma}, u\rangle_\O
  = - \iint_{E_\delta} \widetilde{\sigma} D_\delta u \,\mathrm{d}\lambda_\delta.
\]
Finally, we rescale the conductivities.
To this end, we recall the quantity \(\theta_{\delta,\gamma}\) defined in~\eqref{eq:relative-material-fraction}, which quantifies the ratio between the total available 
conductivity and the conductivity required to assign the maximal value \(\overline\kappa\) to every bond in \(E_\delta\).
In particular, when \(0<\theta_{\delta,\gamma}<1\), it is the fraction of the active bond set that could be assigned the maximal conductivity, whereas \(\theta_{\delta,\gamma}\geq1\) means that all potentially active bonds can be saturated.
Rather than retaining the dimensional conductivity as the design variable, we introduce its occupation fraction 
\[ a(x,x') = \frac{\kappa(x,x')} {\overline{\kappa}}. \] 
The pointwise conductivity bound and the global material constraint then become \( 0\leq a\leq 1\) and \(\int_{E_\delta} a\,\mathrm{d}\lambda_\delta \leq\theta_{\delta,\gamma}. \) 
For \(\vartheta \geq 0\), we therefore define the admissible occupation class 
\begin{equation*} 
  \begin{aligned} 
  \mathcal{A}_{\delta,\vartheta} = 
  \left\{ a\in L^\infty(E_\delta): 
  0\leq a\leq1,\quad 
  \int_{E_\delta} a\,\mathrm{d}\lambda_\delta \leq\vartheta \right\},
  \qquad \mathcal{A}_{\delta,\vartheta}^{\sym}= \mathcal{A}_{\delta,\vartheta}\cap L^\infty_\sym(E_\delta).
\end{aligned} \end{equation*}
As in~\eqref{eq:kappa_adm}, the superscript \({\sym}\) indicates symmetry with 
respect to interchange of the bond endpoints.

Combining all the scalings, we get
\begin{equation*}
    \begin{aligned}
  (\gamma\hat{\rho}_\delta)^{q-1}
  \mathcal{P}_\delta(\overline{\kappa}a,(\gamma\hat{\rho}_{\delta})^{1-q}\widetilde{u})
  &=
  \frac{1}{p}\underbrace{\frac{\overline{\kappa}|E_\delta|}{\gamma}}_{=\frac{|\Od|^2}{\theta_{\delta,\gamma}}}\iint_{E_\delta} a |D_\delta\widetilde{u}|^p\,\mathrm{d}\lambda_\delta
  - \langle f, \widetilde{u} \rangle_\O, \qquad\text{and}\\
  (\gamma\hat{\rho}_\delta)^{q-1} \mathcal{C}_\delta(\overline{\kappa}a,\hat{\rho}_\delta^{-1/p}|E_\delta|^{-1}\widetilde{\sigma}) &= 
  \frac{1}{q}\left(\frac{\gamma}{\overline{\kappa}|E_\delta|}\right)^{q-1} \iint_{E_\delta} a^{1-q} |\widetilde{\sigma}|^q\,\mathrm{d}\lambda_\delta.
    \end{aligned}
\end{equation*}
From this point onward, all state and flux variables are understood in the normalized sense, and we omit the tildes.
For \(\vartheta > 0\) we will thus write
\begin{equation*}
  \begin{aligned}
\mathcal{P}_{\delta,\vartheta}(a,u) &= \frac{1}{p}\frac{|\Od|^2}{\vartheta}
\int_{E_\delta} a|D_\delta u|^p \,\mathrm{d}\lambda_\delta
- \langle f, u \rangle_\O, &\qquad u &\in X_\delta^p,\qquad\text{and}\\
\mathcal{C}_{\delta,\vartheta}(a,\sigma) &= \frac{1}{q}\left(\frac{|\Od|^2}{\vartheta}\right)^{1-q}
\int_{E_\delta} a^{1-q} |\sigma|^q\,\mathrm{d}\lambda_\delta,
&\qquad\sigma &\in Y^q_\delta.
  \end{aligned}
\end{equation*}
Similarly, the rescaled optimal value for the peridynamic free-material optimization problem is
\begin{equation} \label{eq:rescaled_itheta}
 \begin{aligned} i_\delta(\vartheta) &= 
  \inf_{(a,\sigma)\in\mathcal{A}_{\delta,\vartheta}^{\sym}\times Y_\delta^q(f)} \mathcal{C}_{\delta,\vartheta}(a,\sigma) = - \sup_{a\in\mathcal{A}_{\delta,\vartheta}^{\sym}} \inf_{u\in X_\delta^p} \mathcal{P}_{\delta,\vartheta}(a,u). 
\end{aligned}
\end{equation}
For \(\vartheta=\theta_{\delta,\gamma}\), the normalized value is thus related to the original free-material value by \( i_\delta(\theta_{\delta,\gamma}) = (\gamma\widehat\rho_\delta)^{q-1} i(\delta,\gamma,0,\overline\kappa). \)

For the asymptotic analysis, we will consider a prescribed family 
\((\theta_\delta)\subset(0,+\infty)\) and abbreviate \( i_\delta =i_\delta(\theta_\delta)\) and \(\mathcal{A}_\delta^{\sym} = \mathcal{A}_{\delta,\theta_\delta}^{\sym}. \)
All subsequent material variables are occupation densities \(a\), unless a 
conductivity \(\kappa\) is explicitly reintroduced for comparison with the physical
formulation.
Finally, we fix \(\delta_*>0\) sufficiently small and restrict the asymptotic analysis
to \(0<\delta\leq\delta_*\). 
Whenever a common ambient domain is needed, \(Q\) denotes a fixed bounded Lipschitz domain satisfying 
\(Q\supset\operatorname{cl}\O_{\delta_*}. \)

For later use, we record the following identity.
\begin{proposition}\label{prop:Ed_over_Bd}
  \[\lim_{\delta\downarrow 0}\frac{|E_\delta|}{|B_\delta|} = |\O|.\]
\end{proposition}
\begin{proof}
  A change of variables and Tonelli's theorem give: 
  \[\begin{aligned}
  |E_\delta| &= 
  \int_{\Od}\int_{\Od} \mathbf{1}_{B_{\delta}}(x'-x)\,\mathrm{d}x'\,\mathrm{d}x
  =
  \int_{\Od}\int_{\Od-x} \mathbf{1}_{B_{\delta}}(h)\,\mathrm{d}h\,\mathrm{d}x
  \\&= 
  \int_{\Od}\int_{B_\delta} \mathbf{1}_{\Od}(x+h)\,\mathrm{d}h\,\mathrm{d}x=
  \int_{B_\delta}\int_{\Od} \mathbf{1}_{\Od}(x+h)\,\mathrm{d}x\,\mathrm{d}h\\&=
  \int_{B_\delta} |\Od\cap(\Od-h)|\,dh. \end{aligned}\] 
  Owing to the inclusions
  \[\O \subset \Od\cap (\Od-h) \subset \Od,\]
  we have the convergence 
  \[ 
    \frac{|E_\delta|}{|B_\delta|} \to|\O|. \qedhere
  \] 
\end{proof}

Although the state--flux formulation above is restricted to \(1<p<+\infty\), the reduced state energies introduced in Section~\ref{sec:intermediate-material} admit a natural endpoint continuation to \(p=1\); see Subsection~\ref{subsec:p-one}.

%% file: 03_gamma.tex
\section{Reduced state energies and their localization} 
\label{sec:intermediate-material} 
In this section we consider a family of dimensionless material fractions
\begin{equation*} 
  \theta_\delta\to\theta \qquad\text{for some }\theta\in[0,+\infty). 
\end{equation*}
Our aim is to identify the corresponding local limits of the reduced state energies defined below.

\subsection{Elimination of the occupation density} \label{subsec:intermediate-reduction} 
For a nonnegative function \(g \in L^1(E_\delta,\lambda_\delta)\), we denote by \(\mathcal{T}_{\delta,\vartheta}(g)\) its upper-tail functional of mass \(\vartheta \geq 0\):
\begin{equation} \label{eq:nonlocal-upper-tail} 
  \mathcal{T}_{\delta,\vartheta}(g) = \sup_{ a\in \mathcal{A}_{\delta,\vartheta}} 
  \int_{E_\delta}a g\,\mathrm{d}\lambda_\delta, 
\end{equation}
and for \(\vartheta>0\) we let
\begin{equation} \label{eq:intermediate-energy} 
  \mathcal{E}_{\delta,\vartheta}(u) = 
  \begin{cases}\displaystyle\frac{|\O_\delta|^2}{p\vartheta} \mathcal{T}_{\delta,\vartheta} \bigl( |D_\delta u|^p \bigr),& u \in X_\delta^p,\\[2mm]
    +\infty,&\text{otherwise}.\end{cases}
\end{equation} 

\begin{lemma}[Dual representation of the upper-tail functional] \label{lem:upper-tail-duality} 
  Let \((Z,\lambda)\) be a probability space, \(g\in L^1(Z,\lambda)\), \(g\geq0\), and \(\vartheta\geq 0\). 
  Then 
  \begin{equation} \label{eq:dual_tail}
    \sup_{\substack{ a\in L^\infty(Z)\\ 0\leq a\leq1\\ \int_Za\,\mathrm{d}\lambda\leq\vartheta }} \int_Z ag\,\mathrm{d}\lambda 
    = \inf_{t\geq0} \left\{ \vartheta t+ \int_Z(g-t)_+\,\mathrm{d}\lambda \right\}. 
  \end{equation} 
  Moreover, the dual infimum is attained whenever \(\vartheta>0\).
\end{lemma}
\begin{proof} 
  If \(\vartheta=0\), the only admissible occupation is \(a=0\), so the supremum equals zero. Moreover, \[ 0\leq(g-t)_+\leq g \quad\text{and}\quad (g-t)_+\to0 \quad\lambda\text{-a.e. as }t\to+\infty. \] The dominated convergence theorem therefore gives \( \inf_{t\geq0} \int_Z(g-t)_+\,\mathrm d\lambda =0. \)

  Let us now assume that \(\vartheta>0\).
  Let \(t\geq0\) be the Lagrange multiplier associated with the inequality constraint \(\int_Z a\,\mathrm{d}\lambda\leq\vartheta\);
  note that Slater's constraint qualification holds since \(\vartheta>0\).
  The associated Lagrangian is 
  \[
    L(a,t) = \int_Zag\,\mathrm{d}\lambda + t\left( \vartheta-\int_Za\,\mathrm{d}\lambda \right) = \vartheta t+ \int_Za(g-t)\,\mathrm{d}\lambda,
  \]
  Noting that 
  \[ \sup_{0\leq a\leq1}L(a,t) = \vartheta t+ \int_Z(g-t)_+\,\mathrm{d}\lambda \geq \vartheta t, \]
  the claimed equality follows from the strong Lagrangian duality.
  Furthermore, the dual objective is continuous and coercive on \([0,+\infty)\), so the dual problem admits an optimal solution.
\end{proof}

\begin{proposition}[Reduced problem] \label{prop:intermediate-reduction} 
    For every \(\delta >0\) and \(\vartheta>0\)
    \begin{equation} \label{eq:intermediate-reduced-value} 
        i_\delta(\vartheta) = -\inf_{u\in L^p(\O)} \left\{ \mathcal{E}_{\delta,\vartheta}(u) - \langle f,u\rangle_\O \right\}.
    \end{equation}
\end{proposition}
\begin{proof} 
  This is just an application of a convex minimax theorem.
  Let
  \[ \begin{aligned} 
    L_\delta(a,u) &= \frac{|\Od|^2} {p\vartheta} \int_{E_\delta} a|D_\delta u|^p \,\mathrm{d}\lambda_\delta - \langle f,u\rangle_\O,
    \qquad (a,u)\in \mathcal{A}_{\delta,\vartheta}^{s} \times X_\delta^p.
   \end{aligned} \]
   By the saddle-point representation written in the rescaled variables, see~\eqref{eq:rescaled_itheta},
   \[i_\delta(\vartheta) = - \sup_{a\in \mathcal{A}_{\delta,\vartheta}^{s}} \inf_{u\in X_\delta^p} L_\delta(a,u). \]
   We equip \(\mathcal{A}_{\delta,\vartheta}^{s}\) with the weak-star topology of \(L^\infty(E_\delta)\) and \(X_\delta^p\) with its norm topology. 
   The set \(\mathcal{A}_{\delta,\vartheta}^{s}\) is then convex and weak-star compact, and \(X_\delta^p\) is convex.
   For every fixed \(u\in X_\delta^p\), the map \( a\mapsto L_\delta(a,u) \) is affine and weak-star continuous, since \( |D_\delta u|^p\in L^1(E_\delta). \)
   For every fixed \(a\in\mathcal{A}_{\delta,\vartheta}^{s}\), the map \( u\mapsto L_\delta(a,u) \) is convex and continuous on \(X_\delta^p\). 
   Therefore, a convex minimax theorem (see, for example, \cite[Corollary~3.3]{Sio58}) yields 
   \(
   i_\delta(\vartheta) = - \inf_{u\in X_\delta^p}\sup_{a\in \mathcal{A}_{\delta,\vartheta}^{s}}  L_\delta(a,u).
   \)
The symmetry restriction may be dropped in the inner supremum. 
Indeed, for every \(a\in\mathcal{A}_{\delta,\vartheta}\), its symmetric part \( a^{\sym}=\frac{a+a^{T}}2 \) belongs to \(\mathcal{A}_{\delta,\vartheta}^{\sym}\), because \(0\leq a^{\sym}\leq1\) and, by invariance of \(\lambda_\delta\) under endpoint interchange, \( \int_{E_\delta}a^{\sym}\,\mathrm{d}\lambda_\delta = \int_{E_\delta}a\,\mathrm{d}\lambda_\delta. \) Since \(|D_\delta u|^p\) is symmetric, \( \int_{E_\delta} a^{\sym}|D_\delta u|^p \,\mathrm{d}\lambda_\delta = \int_{E_\delta} a|D_\delta u|^p \,\mathrm{d}\lambda_\delta \);
see Lemma~\ref{lem:symmetric-antisymmetric-fields}.
 Consequently, the inner supremum equals \(\mathcal{T}_{\delta,\vartheta}(|D_\delta u|^p)\), which gives~\eqref{eq:intermediate-reduced-value}.
\end{proof}

\subsection{The limiting reduced energies} 
\label{subsec:intermediate-limit-functional} 
Let \( \mathrm{d}\nu(s) = |\mathbb{S}^{d-1}|^{-1} \mathrm{d}\mathcal{H}^{d-1}(s) \) be the uniform probability measure on \(\mathbb{S}^{d-1}\), and let 
\( \mathrm{d}\lambda_0(x,s) = |\O|^{-1}\mathrm{d}\mathcal{L}^d(x) \otimes \mathrm{d}\nu(s) \)
be the uniform probability measure on \(\O\times \mathbb{S}^{d-1}\). 
For \(g\in L^1( \O\times\mathbb{S}^{d-1},\lambda_0)\), \(g\geq0\), define 
\begin{equation} \label{eq:local-upper-tail} 
  \mathcal{T}_\theta(g) = \sup_{\substack{ a\in L^\infty( \O\times\mathbb{S}^{d-1})\\ 0\leq a\leq1\\ \int a\,\mathrm{d}\lambda_0\leq\theta }} 
  \int_{\O\times\mathbb{S}^{d-1}} ag\,\mathrm{d}\lambda_0. 
\end{equation} 
Similarly to~\eqref{eq:intermediate-energy}, for \(\theta>0\) we put 
\begin{equation*} 
  \mathcal{E}_\theta(u) = \begin{cases}\displaystyle\frac{|\O|^2}{p\theta} \mathcal{T}_\theta \bigl( |\nabla u(x)\cdot s|^p \bigr), &
    u \in W^{1,p}_0(\O),\\[2mm]
    +\infty &\text{otherwise},\end{cases}
\end{equation*} 
while for \(\theta = 0\) we define
\begin{equation} \label{eq:intermediate-local-energy0} 
  \mathcal{E}_0(u) = \begin{cases}\displaystyle\frac{|\O|^2}{p}
    \operatorname{Lip}(u)^p, &
    u \in \operatorname{Lip}_0(\O),\\[2mm]
    +\infty &\text{otherwise}.\end{cases}
\end{equation} 
In either case, for \(\theta\in[0,+\infty)\) we let
\begin{equation} \label{eq:intermediate-local-value} 
  i_0(\theta) = -\inf_{u\in L^p(\O)} \left\{ \mathcal{E}_\theta(u) - \langle f,u\rangle_\O \right\}. 
\end{equation}

A direct passage to the limit in~\eqref{eq:prob1} or~\eqref{eq:rescaled_itheta} would require a lower-semicontinuity
result for nonlocal energies with varying two-point coefficients. 
This is related in spirit to the generalized Ponce inequality of~\cite{Mun21}, but the 
present coefficients retain a bond-direction dependence and do not have the averaged 
spatial form considered there. 
The minimax reduction of Proposition~\ref{prop:intermediate-reduction}, followed by the
dual representation of the upper-tail functional, avoids this joint coefficient-state
limit. 
It reduces the lower-bound argument to the nonlocal \(\Gamma\)-liminf inequality for the fixed
convex integrand \( r\mapsto(|r|^p-t)_+. \)

\begin{proposition}[Localization of the reduced energies]
\label{prop:intermediate-gamma-limit}
Assume that
\[
    \theta_\delta\to\theta\in[0,+\infty).
\]
Then
\[
    \mathcal{E}_{\delta,\theta_\delta}
    \xrightarrow{\Gamma}
    \mathcal{E}_\theta
    \qquad\text{in }L^p(\O)
    \quad\text{as }\delta\downarrow0.
\]
Moreover, the family
\((\mathcal{E}_{\delta,\theta_\delta})\)
is equicoercive.
\end{proposition}

\begin{proof}
\textit{Equicoercivity.}
For every nonnegative
\(g\in L^1(E_\delta,\lambda_\delta)\) and every \(\vartheta>0\),
the constant occupation
\(a\equiv\min\{\vartheta,1\}\) is admissible in
\(\mathcal{A}_{\delta,\vartheta}\). Hence, by
\eqref{eq:nonlocal-upper-tail},
\begin{equation*}
    \mathcal{T}_{\delta,\vartheta}(g)
    \geq
    \min\{\vartheta,1\}
    \int_{E_\delta}g\,\mathrm{d}\lambda_\delta.
\end{equation*}
Consequently,
\begin{equation}
\label{eq:etheta_coercivity}
\begin{aligned}
    \mathcal{E}_{\delta,\theta_\delta}(u)
    &\geq
    \underbrace{
        \frac{|\O_\delta|^2}{p}
        \frac{\min\{\theta_\delta,1\}}{\theta_\delta}
    }_{\to\,|\O|^2/[p\max\{\theta,1\}]}
    \int_{E_\delta}|D_\delta u|^p
    \,\mathrm{d}\lambda_\delta.
\end{aligned}
\end{equation}
The nonlocal Poincar\'e inequality and compactness theorem
\cite[Theorems~1.1 and~1.2]{ponce2004estimate} therefore imply
that the family
\((\mathcal{E}_{\delta,\theta_\delta})\)
is equicoercive in \(L^p(\O)\).

\medskip
\noindent
\textit{The liminf inequality.}
Let \(\delta_k\downarrow0\) and let
\(u_{\delta_k}\to u\) strongly in \(L^p(\O)\).
After passing to a subsequence realizing the lower limit, we may
assume that
\[
    \sup_k
    \mathcal{E}_{\delta_k,\theta_{\delta_k}}
        (u_{\delta_k})
    <+\infty,
\]
since otherwise there is nothing to prove.
By \eqref{eq:etheta_coercivity} and the nonlocal characterization
of Sobolev spaces,
\(
    u\in W^{1,p}_0(\O).
\)

For every \(k\), let \(t_{\delta_k}\geq0\) be a minimizer in
\eqref{eq:dual_tail}, corresponding to
\(
    g=|D_{\delta_k}u_{\delta_k}|^p.
\)
For \(r\geq0\), let
\(
    \omega_t(r)=(r^p-t)_+.
\)
Lemma~\ref{lem:upper-tail-duality} gives
\begin{equation}
\label{eq:gammalinf_bound}
\begin{aligned}
    \mathcal{E}_{\delta_k,\theta_{\delta_k}}
        (u_{\delta_k})
    &=
    \frac{|\O_{\delta_k}|^2}{p}
    \left[
        t_{\delta_k}
        +
        \frac{1}{\theta_{\delta_k}}
        \int_{E_{\delta_k}}
        \omega_{t_{\delta_k}}\bigl(
            |D_{\delta_k}u_{\delta_k}|
        \bigr)
        \,\mathrm{d}\lambda_{\delta_k}
    \right]
    \geq
    \frac{|\O_{\delta_k}|^2}{p}t_{\delta_k}.
\end{aligned}
\end{equation}
Thus \((t_{\delta_k})\) is bounded. After passing to a further
subsequence, we may assume that
\(
    t_{\delta_k}\to t
\)
for some \(t\geq0\).

Let
\(
    \rho_{\delta_k}(h)
    =
    |B_{\delta_k}|^{-1}
    \mathbf{1}_{B_{\delta_k}}(h).
\)
Then
\[
\begin{aligned}
    \int_{E_{\delta_k}}
        \omega_t(|D_{\delta_k}u_{\delta_k}|)
        \,\mathrm{d}\lambda_{\delta_k}
    &=
    \frac{|B_{\delta_k}|}{|E_{\delta_k}|}
    \int_Q\int_Q
        \omega_t\left(
            \frac{
                |u_{\delta_k}(x)-u_{\delta_k}(x')|
            }{
                |x-x'|
            }
        \right)
        \rho_{\delta_k}(x-x')
        \,\mathrm{d}x'\,\mathrm{d}x,
\end{aligned}
\]
where
\(Q\supset\operatorname{cl}\O_{\delta_*}\)
is the fixed bounded Lipschitz domain.

The map \(r\mapsto\omega_t(r)\) is continuous, nonnegative,
nondecreasing, and convex, and satisfies
\[
    0\leq\omega_t(r)\leq r^p.
\]
The nonlocal liminf inequality
\cite[Theorem~8 and Corollary~8]{ponce2004new},
the radial-kernel identification
\cite[Example~1]{ponce2004new}, and
\[
    \frac{|B_{\delta_k}|}{|E_{\delta_k}|}
    \to
    \frac1{|\O|},
\]
see Proposition~\ref{prop:Ed_over_Bd}, therefore give
\begin{equation}
\label{eq:gammalinf_ponce_impl}
\begin{aligned}
    \liminf_{k\to\infty}
    \int_{E_{\delta_k}}
        \omega_{t_{\delta_k}}
        (|D_{\delta_k}u_{\delta_k}|)
        \,\mathrm{d}\lambda_{\delta_k}
    &=
    \liminf_{k\to\infty}
    \int_{E_{\delta_k}}
        \omega_t
        (|D_{\delta_k}u_{\delta_k}|)
        \,\mathrm{d}\lambda_{\delta_k}
\\
    &\geq
    \frac1{|\O|}
    \int_\O
    \int_{\mathbb{S}^{d-1}}
        \omega_t(|\nabla u(x)\cdot s|)
        \,\mathrm{d}\nu(s)\,\mathrm{d}x
\\
    &=
    \int_{\O\times\mathbb{S}^{d-1}}
        \omega_t(|\nabla u(x)\cdot s|)
        \,\mathrm{d}\lambda_0(x,s).
\end{aligned}
\end{equation}
Here the first equality follows from
\[
    |\omega_{t_{\delta_k}}(r)-\omega_t(r)|
    \leq
    |t_{\delta_k}-t|,
    \qquad r\geq0.
\]

\smallskip
\noindent
\emph{Case \(\theta>0\).}
Using \eqref{eq:gammalinf_bound},
\eqref{eq:gammalinf_ponce_impl}, and
\(\theta_{\delta_k}\to\theta\), we obtain
\[
    \lim_{k\to\infty}
    \mathcal{E}_{\delta_k,\theta_{\delta_k}}
        (u_{\delta_k})
    \geq
    \frac{|\O|^2}{p\theta}
    \left[
        \theta t
        +
        \int_{\O\times\mathbb{S}^{d-1}}
            \omega_t\bigl(
                |\nabla u\cdot s|
            \bigr)
            \,\mathrm{d}\lambda_0
    \right]
\geq
    \mathcal{E}_\theta(u),
\]
where the last inequality follows from
Lemma~\ref{lem:upper-tail-duality}.

\smallskip
\noindent
\emph{Case \(\theta=0\).}
The boundedness in \eqref{eq:gammalinf_bound} gives
\[
    0
    \leq
    \int_{E_{\delta_k}}
        \omega_{t_{\delta_k}}
        (|D_{\delta_k}u_{\delta_k}|)
        \,\mathrm{d}\lambda_{\delta_k}
    \leq
    \frac{p\theta_{\delta_k}}
         {|\O_{\delta_k}|^2}
    \mathcal{E}_{\delta_k,\theta_{\delta_k}}
        (u_{\delta_k})
    \to 0.
\]
Together with \eqref{eq:gammalinf_ponce_impl}, this yields
\[
    \int_{\O\times\mathbb{S}^{d-1}}
        \omega_t\bigl(
            |\nabla u(x)\cdot s|
        \bigr)
        \,\mathrm{d}\lambda_0(x,s)
    =0.
\]
Hence
\(
    |\nabla u(x)\cdot s|^p\leq t
\)
for \(\lambda_0\)-almost every \((x,s)\).
For almost every \(x\), continuity with respect to \(s\) extends
the inequality to every \(s\in\mathbb{S}^{d-1}\). Therefore,
\[
    |\nabla u(x)|^p
    =
    \sup_{s\in\mathbb{S}^{d-1}}
        |\nabla u(x)\cdot s|^p
    \leq t
\]
for almost every \(x\in\O\). Thus the zero extension of \(u\)
has a representative in \(\operatorname{Lip}_0(\O)\), and
\[
    \operatorname{Lip}(u)^p
    =
    \|\nabla u\|_{L^\infty(\O;\R^d)}^p
    \leq t.
\]
Using \eqref{eq:gammalinf_bound} once more, we conclude that
\begin{equation*}
\begin{aligned}
    \liminf_{k\to\infty}
    \mathcal{E}_{\delta_k,\theta_{\delta_k}}
        (u_{\delta_k})
    \geq
    \lim_{k\to\infty}
    \frac{|\O_{\delta_k}|^2}{p}t_{\delta_k}
    =
    \frac{|\O|^2}{p}t
    \geq
    \frac{|\O|^2}{p}\operatorname{Lip}(u)^p
    =
    \mathcal{E}_0(u).
\end{aligned}
\end{equation*}
This proves the liminf inequality.

\medskip
\noindent
\textit{The limsup inequality.}
We construct constant recovery sequences in both cases.

\smallskip
\noindent
\emph{Case \(\theta>0\).}
Let \(u\in W^{1,p}_0(\O)\), and, for \(t\geq0\), define
\[
\begin{aligned}
    F_\delta(t)
    &=
    \theta_\delta t
    +
    \int_{E_\delta}
        \omega_t(|D_\delta u|)
        \,\mathrm{d}\lambda_\delta,
\\
    F(t)
    &=
    \theta t
    +
    \int_{\O\times\mathbb{S}^{d-1}}
        \omega_t(|\nabla u(x)\cdot s|)
        \,\mathrm{d}\lambda_0(x,s).
\end{aligned}
\]
The zero extension \(u_0\) of \(u\) belongs to
\(W^{1,p}(\Rn)\). Since \(\omega_t\) is continuous and
nonnegative and satisfies
\[
    0\leq\omega_t(r)\leq r^p,
\]
the pointwise nonlocal-to-local convergence theorem
\cite[Theorem~4]{ponce2004new}, together with the radial-kernel
identification \cite[Example~1]{ponce2004new}, gives
\[
    F_\delta(t)\to F(t)
    \qquad\text{for every }t\geq0.
\]

Let \(t_\delta\) minimize \(F_\delta\). Since
\(F_\delta(t_\delta)\leq F_\delta(0)\),
\[
    \theta_\delta t_\delta
    \leq
    \int_{E_\delta}|D_\delta u|^p
    \,\mathrm{d}\lambda_\delta.
\]
The right-hand side remains bounded, and
\(\theta_\delta\to\theta>0\). Hence the minimizing thresholds
\(t_\delta\) belong to a common compact interval \([0,M]\).
Moreover,
\[
    |F_\delta(t)-F_\delta(\widetilde t)|
    \leq
    (1+\theta_\delta)|t-\widetilde t|,
    \qquad t,\widetilde t\geq0.
\]
Thus \(F_\delta\to F\) uniformly on \([0,M]\), and consequently
\[
    \mathcal{T}_{\delta,\theta_\delta}
        (|D_\delta u|^p)
    =
    \min_{t\geq0}F_\delta(t)
    \to
    \min_{t\geq0}F(t)
    =
    \mathcal{T}_\theta
        (|\nabla u\cdot s|^p),
\]
see Lemma~\ref{lem:upper-tail-duality}. Therefore,
\[
    \mathcal{E}_{\delta,\theta_\delta}(u)
    \to
    \mathcal{E}_\theta(u).
\]
If \(u\notin W^{1,p}_0(\O)\), then
\(\mathcal{E}_\theta(u)=+\infty\), and the limsup condition is
automatic.

\smallskip
\noindent
\emph{Case \(\theta=0\).}
Let \(u\in\operatorname{Lip}_0(\O)\) and use the constant
sequence \(u_\delta=u\). Since the zero extension of \(u\) is
globally Lipschitz,
\[
    |D_\delta u(x,x')|
    \leq
    \operatorname{Lip}(u)
\]
for almost every \((x,x')\in E_\delta\). Moreover, for every
nonnegative \(g\in L^\infty(E_\delta)\),
\[
    \mathcal{T}_{\delta,\theta_\delta}(g)
    \leq
    \theta_\delta
    \|g\|_{L^\infty(E_\delta)}.
\]
Consequently,
\[
    \mathcal{E}_{\delta,\theta_\delta}(u)
    \leq
    \frac{|\Od|^2}{p}
    \|D_\delta u\|_{L^\infty(E_\delta)}^p
    \leq
    \frac{|\Od|^2}{p}
    \operatorname{Lip}(u)^p.
\]
Since \(|\O_\delta|\to|\O|\),
\[
    \limsup_{\delta\downarrow0}
    \mathcal{E}_{\delta,\theta_\delta}(u)
    \leq
    \mathcal{E}_0(u).
\]
If \(u\notin\operatorname{Lip}_0(\O)\), then
\(\mathcal{E}_0(u)=+\infty\), and the limsup condition is
automatic.
\end{proof}

\begin{remark}[Existence of reduced optimal states]
We note that the reduced problem~\eqref{eq:intermediate-reduced-value} attains its infimum.
Indeed, the estimate~\eqref{eq:etheta_coercivity} used with \(\vartheta=\theta_\delta>0\), together with the nonlocal Poincar{\'e} inequality, shows that 
\[ u\mapsto \mathcal{E}_{\delta,\theta_\delta}(u) - \langle f,u\rangle_\O \] is coercive on \(X_\delta^p\). 
Furthermore, by~\eqref{eq:nonlocal-upper-tail}, 
\(\mathcal{E}_{\delta,\theta_\delta}\) is the supremum, over all admissible densities \(a\), of the convex weakly lower
semicontinuous functionals 
\[ u\mapsto \frac{|\Od|^2} {p\theta_\delta} \int_{E_\delta} a|D_\delta u|^p \,\mathrm{d}\lambda_\delta. \] 
It is therefore weakly lower semicontinuous on \(X_\delta^p\). Since \(X_\delta^p\) is reflexive and \(u\mapsto\langle f,u\rangle_\O\) is continuous on \(X_\delta^p\), the direct method yields a minimizer.
\end{remark}

\begin{theorem}[Localization of optimal values and states] \label{thm:intermediate-localization} 
  Assume that 
  \[ \theta_\delta\to\theta \qquad\text{for some }\theta\in[0,+\infty). \] 
  Then 
  \[ i_\delta(\theta_\delta) \to i_0(\theta), \] 
  where \(i_0(\theta)\) is defined in~\eqref{eq:intermediate-local-value}. 
  Moreover, if \(u_\delta\) is any optimal state for~\eqref{eq:intermediate-reduced-value}, then \((u_\delta)\) is relatively compact in \(L^p(\O)\), and every cluster point is an optimal state for~\eqref{eq:intermediate-local-value}. 
  If the latter state is unique, then the whole family converges. 
\end{theorem} 
\begin{proof} By Proposition~\ref{prop:intermediate-reduction}, \[ i_\delta(\theta_\delta) = - \inf_{u\in L^p(\O)} \left\{ \mathcal{E}_{\delta,\theta_\delta}(u) - \langle f,u\rangle_\O \right\}. \] Since \( u\mapsto-\langle f,u\rangle_\O \) is continuous on \(L^p(\O)\), Proposition~\ref{prop:intermediate-gamma-limit} implies that \[ \mathcal{E}_{\delta,\theta_\delta} - \langle f,\cdot\rangle_\O \xrightarrow{\Gamma} \mathcal{E}_\theta - \langle f,\cdot\rangle_\O \] in \(L^p(\O)\). Moreover, estimate~\eqref{eq:etheta_coercivity} and the nonlocal Poincar{\'e} inequality give, uniformly for sufficiently small \(\delta\), \[ \mathcal{E}_{\delta,\theta_\delta}(u) \geq c\|u\|_{L^p(\O)}^p. \] Consequently, by H{\"o}lder's and Young's inequalities, \[ \mathcal{E}_{\delta,\theta_\delta}(u) - \langle f,u\rangle_\O \geq c'\|u\|_{L^p(\O)}^p-C, \] for constants \(c',C>0\) independent of sufficiently small \(\delta\). Thus the perturbed functionals are equicoercive. The assertions now follow from the fundamental theorem of \(\Gamma\)-convergence. \end{proof}

\subsection{The sparse endpoint \(\theta=0\)} 
The limiting reduced energies are pointwise continuous at the sparse endpoint \(\theta=0\).
Indeed, if \(g\geq0\) belongs to \(L^1(\O\times\mathbb{S}^{d-1},\lambda_0)\), then 
\[ \lim_{\theta\downarrow0} \frac{1}{\theta}\mathcal{T}_\theta(g) = 
\operatorname*{ess\,sup}_{\O\times\mathbb{S}^{d-1}} g. \] 
Consequently, for every \(u\in L^p(\O)\), 
\[ \lim_{\theta\downarrow0}\mathcal{E}_\theta(u) = \mathcal{E}_0(u). \] 
Indeed, for \(u\in W_0^{1,p}(\O)\), 
\[ \begin{aligned} 
  \lim_{\theta\downarrow0}\mathcal{E}_\theta(u) &= \frac{|\O|^2}{p} 
  \operatorname*{ess\,sup}_{(x,s) \in\O\times\mathbb{S}^{d-1}} |\nabla u(x)\cdot s|^p = \frac{|\O|^2}{p} \|\nabla u\|_{L^\infty(\O;\mathbb{R}^d)}^p. \end{aligned} \] 
This quantity is finite precisely when the zero extension of \(u\) belongs to \(\operatorname{Lip}_0(\O)\), in which case 
\[ \|\nabla u\|_{L^\infty(\O;\Rn)} = \operatorname{Lip}(u). \] 
Thus the definition of \(\mathcal{E}_0\) in \eqref{eq:intermediate-local-energy0} is the pointwise continuation 
of the upper-tail family \((\mathcal{E}_\theta)_{\theta>0}\) at \(\theta=0\). 
Proposition~\ref{prop:intermediate-gamma-limit} strengthens this pointwise observation 
by allowing the material fraction and the interaction horizon to vanish simultaneously.

\subsection{The saturated range \(\theta\geq 1\)} \label{subsec:saturated}

Whenever \(\vartheta\geq1\), the occupation \(a\equiv1\) is admissible in 
\(\mathcal{A}_{\delta,\vartheta}\), and therefore 
\[ \mathcal{T}_{\delta,\vartheta}(g) = \int_{E_\delta}g\,\mathrm{d}\lambda_\delta. \] 
It follows that 
\[ \mathcal{E}_{\delta,\vartheta}(u) = \frac{|\Od|^2}{p\vartheta} 
\int_{E_\delta}|D_\delta u|^p \,\mathrm{d}\lambda_\delta, \qquad u\in X_\delta^p. 
\] 
Thus, whenever \(\theta_\delta\geq1\), the material optimization is inactive and the reduced 
problem is precisely a rescaled nonlocal \(p\)-Dirichlet problem.

Similar reasoning applies to the limiting local problem.
Therefore, if \[ \theta_\delta\to\theta\in[1,+\infty), \] then 
Proposition~\ref{prop:intermediate-gamma-limit} and Theorem~\ref{thm:intermediate-localization} apply with 
\[ \mathcal{E}_\theta(u) = \begin{cases} \displaystyle \frac{K_{d,p}|\O|}{p\theta} \int_\O|\nabla u|^p\,\mathrm dx, &u\in W_0^{1,p}(\O), \\[1ex] +\infty, &\text{otherwise}. \end{cases} \] 
In particular, 
\[ i_\delta(\theta_\delta)\to i_0(\theta). \]
Since \[ u \mapsto \mathcal{E}_\theta(u)-\langle f,u\rangle_\O \] is strictly convex, its minimizer is unique, and the entire sequence of optimal reduced states converges in \(L^p(\O)\).

\subsection{The endpoint \(p=1\)} \label{subsec:p-one} 

Although the state--flux formulation of Section~\ref{sec:setup} was introduced under the assumption \(1<p<+\infty\), the reduced state energies admit a natural continuation to \(p=1\). 
In this subsection we briefly record the corresponding localization result.
We use the normalized difference operator \(D_\delta\) and the probability measure \(\lambda_\delta\) introduced above, and set 
\[ X_\delta^1 = \{u\in L^1(\O): D_\delta u\in L^1(E_\delta,\lambda_\delta)\}, \qquad \|u\|_{X_\delta^1} = \|D_\delta u\|_{L^1(E_\delta,\lambda_\delta)}.
\] 
The endpoint reduced energy \(\mathcal{E}_{\delta,\vartheta}^{(1)}\) is defined by setting \(p=1\) in~\eqref{eq:intermediate-energy}.
Lemma~\ref{lem:upper-tail-duality} applies without change when \(p=1\).

Let \(u_0\) denote the extension of \(u\) by zero outside \(\O\), regarded as a function on the fixed ambient domain \(Q\) introduced in Subsection~\ref{subsec:rescaling}. 
For \(u_0\in BV(Q)\), we write 
\[ 
  Du_0=\nabla u\,\mathcal{L}^d\mres\O +D^s u_0 . 
\] 
Thus \(D^s u_0\) includes any singular contribution concentrated on \(\partial\O\).
For \(\vartheta>0\), define 
\[ \mathcal{E}_\vartheta^{(1)}(u) = 
\begin{cases} \displaystyle \frac{|\O|^2}{\vartheta} \left[ T_\vartheta\bigl(|\nabla u(x)\cdot s|\bigr) + \frac{K_{d,1}}{|\O|} |D^s u_0|(Q) \right], & u_0\in BV(Q),\\[2ex] +\infty, & \text{otherwise}, \end{cases} 
\] 
where \(T_\vartheta\) is the upper-tail functional from~\eqref{eq:local-upper-tail}. 
Equivalently, 
\[ 
  \mathcal{E}_\vartheta^{(1)}(u) = 
  |\O|^2 \inf_{t\geq0} \left\{ t + \frac{1}{\vartheta|\O|} \int_\O\int_{\mathbb{S}^{d-1}} \bigl(|\nabla u(x)\cdot s|-t\bigr)_+ \,d\nu(s)\dx + 
  \frac{K_{d,1}}{\vartheta|\O|} |D^s u_0|(Q) \right\}.
\] 
At the sparse endpoint \(\vartheta=0\), we use the definition~\eqref{eq:intermediate-local-energy0} with \(p=1\) and put \(\mathcal{E}_0^{(1)}(u) = \mathcal{E}_0(u)\).

\begin{proposition}[Localization at \(p=1\)] \label{prop:localization-p-one} 
  Assume that \(\theta_\delta\to\theta\in[0,+\infty)\). 
  Then 
  \[ \mathcal{E}_{\delta,\theta_\delta}^{(1)} \xrightarrow{\Gamma} \mathcal{E}_\theta^{(1)} \qquad\text{in }L^1(\O)
    \quad\text{as }\delta\downarrow0. 
  \] 
  Moreover, the family \(\bigl(\mathcal{E}_{\delta,\theta_\delta}^{(1)}\bigr)\) is equicoercive in \(L^1(\O)\). 
\end{proposition} 
\begin{proof} 
  The argument follows the proof of Proposition~\ref{prop:intermediate-gamma-limit}, and we indicate only the modifications required for \(p=1\).
  
  We obtain the estimate~\eqref{eq:etheta_coercivity} with \(p=1\) using the same arguments as in the proof of Proposition~\ref{prop:intermediate-gamma-limit}.
  This estimate implies equicoercivity in \(L^1(\O)\)
  owing to the nonlocal Poincar{\'e} inequality and the compactness theorem~\cite[Theorems~1.1 and~1.2]{ponce2004estimate}.

  We then follow the proof of Proposition~\ref{prop:intermediate-gamma-limit} to obtain the liminf-inequality.
  In particular, we obtain the boundedness of \(t_{\delta_k}\) as in~\eqref{eq:gammalinf_bound}, and switch to a subsequence along which \(t_{\delta_k}\to t\geq0\).
  Boundedness of energies implies compactness in \(L^1(Q)\).
  After passing to a subsequence, the zero extensions converge in \(L^1(Q)\) to some \(u_0\in BV(Q)\).
  Thus we apply the nonlocal \(\Gamma\)-liminf inequality~\cite[Theorem~8 and Corollary~8]{ponce2004new} to the convex integrand \( \omega_t(r)=(r-t)_+, \) which has linear growth and an asymptotic slope \( \omega_{t,\infty}=\lim_{r\to\infty} \omega_t(r)/r = 1\).
  Therefore its recession function is \(r\mapsto r\).
  Consequently, the absolutely continuous part of \(Du_0\) produces the term 
  \[ \frac{1}{|\O|} \int_\O\int_{\mathbb{S}^{d-1}} \bigl(|\nabla u(x)\cdot s|-t\bigr)_+ \,d\nu(s)\dx, 
  \] 
  whereas the singular part produces the term
  \[ \frac{K_{d,1}}{|\O|}|D^s u_0|(Q). \] 
  Taking an infimum with respect to \(t\) gives the asserted lower bound when \(\theta>0\). 
  If \(\theta=0\), the same argument gives 
  \[ \int_\O\int_{\mathbb{S}^{d-1}} \bigl(|\nabla u(x)\cdot s|-t\bigr)_+ \,d\nu(s)\dx + K_{d,1}|D^s u_0|(Q) =0. 
  \] 
  Therefore \(D^s u_0=0\) and \(|\nabla u(x)\cdot s|\leq t\) for almost every \((x,s)\). 
  It follows that \(u_0\) has a representative in 
  \(\operatorname{Lip}_0(\O)\) with \(\operatorname{Lip}(u)\leq t\). 
  Using~\eqref{eq:gammalinf_bound} with \(p=1\), exactly as in Proposition~\ref{prop:intermediate-gamma-limit}, now yields the liminf inequality.

  For the upper bound when \(\theta>0\), the pointwise nonlocal-to-local convergence formula \cite[Theorem~3]{ponce2004new}, together with the radial-kernel identification \cite[Example~1]{ponce2004new}, applies to \(\omega_t(r)=(r-t)_+\) and gives the preceding absolutely continuous and singular terms. 
  The boundedness and equicontinuity argument for the minimizing thresholds is unchanged.
  When \(\theta=0\), the constant recovery sequence arguments used in Proposition~\ref{prop:intermediate-gamma-limit} apply verbatim to every \(u\in\operatorname{Lip}_0(\O)\).
\end{proof}

%% file: 04_flux_recovery.tex
\section{Recovery of optimal fluxes in the sparse regime}
\label{sec:sparse-flux-recovery}

The state formulation developed in Section~\ref{sec:intermediate-material}
already identifies the sparse limit, corresponding to \(\theta=0\),
of the normalized free-material problem.
The purpose of this section is different: we show that optimal two-point
fluxes of the nonlocal problems recover optimal measure-valued fluxes for the
local optimal design problem. Throughout this section, we assume that
\[
    \theta_\delta\to0
    \qquad\text{as }\delta\downarrow0.
\]

\subsection{The local optimal design problem}
\label{subsec:local-optimal-layout}

Let
\begin{equation}
\label{eq:m0-local-state}
    m_0
    =
    \max_{\substack{
        v\in\operatorname{Lip}_0(\O)\\
        \operatorname{Lip}(v)\leq1
    }}
    \langle f,v\rangle_\O.
\end{equation}
The maximum is attained by the Arzel{\`a}--Ascoli theorem. 
Moreover, the standard duality arguments yield
\begin{equation}
\label{eq:m0-local-flux}
    m_0
    =
    \inf_{\substack{
        \sigma\in L^1(\O;\Rn)\\
        \operatorname{div}\sigma=f
    }}
    \|\sigma\|_{L^1(\O;\Rn)}
    =
    \min_{\substack{
        \mu\in\mathcal{M}(\operatorname{cl}\O;\Rn)\\
        \operatorname{div}\mu=f
    }}
    \|\mu\|_{\mathcal{M}(\operatorname{cl}\O;\Rn)},
\end{equation}
see Appendix~\ref{appx:loc_rlx}.
Here the divergence constraint is understood in \(\mathcal{D}'(\O)\).

The endpoint value in \eqref{eq:intermediate-local-value}
can be expressed in terms of \(m_0\). 
Indeed, using the positive homogeneity of the Lipschitz seminorm, 
we obtain
\begin{equation*}
\begin{aligned}
    i_0(0)
    &=-
    \inf_{u\in\operatorname{Lip}_0(\O)}
    \left\{
        \frac{|\O|^2}{p}\operatorname{Lip}(u)^p
        -\langle f,u\rangle_\O
    \right\}
    =-
    \inf_{r\geq0}
    \left\{
        \frac{|\O|^2}{p}r^p-rm_0
    \right\}
    =
    \frac{m_0^q}{q|\O|^{2(q-1)}}.
\end{aligned}
\end{equation*}
Consequently, Theorem~\ref{thm:intermediate-localization} implies
\begin{equation}
\label{eq:sparse-value-convergence}
    i_\delta(\theta_\delta)
    \to
    \frac{m_0^q}{q|\O|^{2(q-1)}}.
\end{equation}

\subsection{The normalized \texorpdfstring{\(L^1\)}{L1}-divergence and flux recovery}
\label{subsec:normalized-flux-recovery}

We use the normalized bond measure \(\lambda_\delta\) introduced in
Subsection~\ref{subsec:rescaling}.
Let \(Y_\delta^1\) denote the domain of the graph closure of 
\[ 
  \operatorname{div}_\delta: Y_\delta^q\subset L^1(E_\delta,\lambda_\delta) \to L^q(\O), 
\] 
and continue to denote the closed realization by \(\operatorname{div}_\delta\). 
We set 
\(
  Y_\delta^1(f) = \{\sigma\in Y_\delta^1: \operatorname{div}_\delta\sigma=f\}; 
\) 
see Appendix~\ref{app:uncapped-exact-localization} for details.
For every \(\sigma\in Y^1_\delta\) and
\(v\in\operatorname{Lip}_0(\O)\), the closed realization satisfies
\begin{equation}
\label{eq:L1-normalized-ibp}
    \langle\operatorname{div}_\delta\sigma,v\rangle_\O
    =
    -\int_{E_\delta}\sigma D_\delta v\,
    \mathrm d\lambda_\delta.
\end{equation}
Indeed, the identity follows by passing to the limit along a graph-approximating sequence, since \[ \|D_\delta v\|_{L^\infty(E_\delta)} \leq \operatorname{Lip}(v). \]
As before, we may restrict our attention to antisymmetric fluxes;
see Lemma~\ref{lem:symmetric-antisymmetric-fields}.

For \(\sigma\in L^1(E_\delta,\lambda_\delta)\), define the recovered
one-point vector field by
\begin{equation}
\label{eq:normalized-recovery-operator}
    R_\delta\sigma(x)
    =
    \frac1{|E_\delta|}
    \int_{\Od}
        \frac{x-x'}{|x-x'|}\,
        \sigma(x,x')
        \mathbf{1}_{B_\delta}(x-x')
        \dx',
    \qquad x\in\Od.
\end{equation}
The normalization in \eqref{eq:normalized-recovery-operator} is the one
corresponding to the reference measure \(\lambda_\delta\). In particular,
Fubini's theorem gives
\begin{equation}
\label{eq:recovery-L1-bound}
    \|R_\delta\sigma\|_{L^1(\Od;\Rn)}
    \leq
    \|\sigma\|_{L^1(E_\delta,\lambda_\delta)}.
\end{equation}

\begin{proposition}[Compactness and consistency of recovered fluxes]
\label{prop:normalized-flux-recovery}
Let \(\delta_k\downarrow0\), and let
\(\sigma_k\in Y^1_{\delta_k}(f)\) be antisymmetric and satisfy
\(
    \sup_k
    \|\sigma_k\|_{L^1(E_{\delta_k},\lambda_{\delta_k})}
    <+\infty.
\)
Set
\(
    \mu_k
    =
    R_{\delta_k}\sigma_k\,\mathcal{L}^d,
\)
where we extend the densities by zero outside \(\O_{\delta_k}\).
Then \((\mu_k)\) is relatively weak-star compact in
\(\mathcal{M}(\operatorname{cl}\O_{\delta_*};\Rn)\).
Every weak-star cluster point
\(\mu\in\mathcal{M}(\operatorname{cl}\O;\Rn)\) satisfies
\[
    \operatorname{div}\mu=f
    \qquad\text{in }\mathcal{D}'(\O),
\]
and
\[
    \|\mu\|_{\mathcal{M}(\operatorname{cl}\O;\Rn)}
    \leq
    \liminf_{k\to\infty}
    \|\sigma_k\|_{L^1(E_{\delta_k},\lambda_{\delta_k})}.
\]
In particular, if
\[
    \|\sigma_k\|_{L^1(E_{\delta_k},\lambda_{\delta_k})}
    \to m_0,
\]
then every cluster point is an optimal solution of the measure problem
in \eqref{eq:m0-local-flux}.
\end{proposition}

\begin{proof}
Estimate~\eqref{eq:recovery-L1-bound} and the Banach--Alaoglu theorem give
weak-star compactness. 
Since
\(\operatorname{spt}\mu_k\subset\operatorname{cl}\O_{\delta_k}\),
every cluster point is supported in \(\operatorname{cl}\O\).

Let \(\varphi\in C_c^\infty(\O)\). By
\eqref{eq:L1-normalized-ibp},
\[
    \langle f,\varphi\rangle_\O
    =
    -\int_{E_{\delta_k}}
        \sigma_k D_{\delta_k}\varphi
        \,\mathrm d\lambda_{\delta_k}.
\]
On the other hand, by \eqref{eq:normalized-recovery-operator},
\[
\begin{aligned}
    \int_{\O_{\delta_k}}
        \nabla\varphi(x)\cdot R_{\delta_k}\sigma_k(x)
        \dx
    =
    \int_{E_{\delta_k}}
        \sigma_k(x,x')
        \nabla\varphi(x)\cdot
        \frac{x-x'}{|x-x'|}
        \,\mathrm{d}\lambda_{\delta_k}(x,x').
\end{aligned}
\]
Therefore,
\[
\begin{aligned}
    \left|
        \int_{\O_{\delta_k}}
            \nabla\varphi\cdot R_{\delta_k}\sigma_k
            \dx
        +\langle f,\varphi\rangle_\O
    \right|
    \leq
    \omega_{\nabla\varphi}(\delta_k)
    \|\sigma_k\|_{L^1(E_{\delta_k},\lambda_{\delta_k})},
\end{aligned}
\]
where
\[
    \omega_{\nabla\varphi}(r)
    =
    \sup_{|x-y|\leq r}
    |\nabla\varphi(x)-\nabla\varphi(y)|.
\]
The right-hand side tends to zero. Passing to a weak-star convergent
subsequence gives
\[
    -\int_{\operatorname{cl}\O}\nabla\varphi\cdot\mathrm d\mu
    =
    \langle f,\varphi\rangle_\O.
\]
Thus \(\operatorname{div}\mu=f\). 
The norm estimate follows from~\eqref{eq:recovery-L1-bound} and weak-star lower
semicontinuity of total variation.
The final assertion follows from~\eqref{eq:m0-local-flux}.
\end{proof}

\subsection{Optimal two-point fluxes at the sparse endpoint}
\label{subsec:optimal-capped-sparse-fluxes}

We now apply Proposition~\ref{prop:normalized-flux-recovery} to optimal
fluxes of the nonlocal problem. 
Let \((a_\delta,\sigma_\delta)\) be an optimal pair in the normalized
complementary formulation \eqref{eq:rescaled_itheta}.
By Proposition~\ref{prop:existence_varpinc}, we may choose \(\sigma_\delta\) to be antisymmetric. 
Then
\begin{equation}
\label{eq:optimal-complementary-energy-sparse}
    i_\delta(\theta_\delta)
    =
    \frac{1}{q}
    \left(
        \frac{|\Od|^2}{\theta_\delta}
    \right)^{1-q}
    \int_{E_\delta}
        a_\delta^{1-q}|\sigma_\delta|^q
        \,\mathrm{d}\lambda_\delta.
\end{equation}
By H\"older's inequality and
\(\int_{E_\delta}a_\delta\,\mathrm d\lambda_\delta\leq\theta_\delta\),
\[
\begin{aligned}
    \|\sigma_\delta\|_{L^1(E_\delta,\lambda_\delta)}
    &\leq
    \left(
        \int_{E_\delta}
            a_\delta^{1-q}|\sigma_\delta|^q
            \,\mathrm d\lambda_\delta
    \right)^{1/q}
    \left(
        \int_{E_\delta}a_\delta
        \,\mathrm{d}\lambda_\delta
    \right)^{1/p}
    \leq
    \theta_\delta^{1/p}
    \left(
        \int_{E_\delta}
            a_\delta^{1-q}|\sigma_\delta|^q
            \,\mathrm{d}\lambda_\delta
    \right)^{1/q}.
\end{aligned}
\]
Together with \eqref{eq:optimal-complementary-energy-sparse}, this gives
\begin{equation}
\label{eq:sparse-flux-upper-bound}
    \frac{
        \|\sigma_\delta\|_{L^1(E_\delta,\lambda_\delta)}^q
    }{
        q|\Od|^{2(q-1)}
    }
    \leq
    i_\delta(\theta_\delta).
\end{equation}

Conversely, if
\(v\in\operatorname{Lip}_0(\O)\) and
\(\operatorname{Lip}(v)\leq1\), then
\eqref{eq:L1-normalized-ibp} gives
\[
\begin{aligned}
    \langle f,v\rangle_\O
    &=-
    \int_{E_\delta}
        \sigma_\delta D_\delta v
        \,\mathrm{d}\lambda_\delta
    \leq
    \|\sigma_\delta\|_{L^1(E_\delta,\lambda_\delta)}
    \|D_\delta v\|_{L^\infty(E_\delta)}
    \leq
    \|\sigma_\delta\|_{L^1(E_\delta,\lambda_\delta)}.
\end{aligned}
\]
Taking the supremum over all such \(v\) yields
\begin{equation}
\label{eq:sparse-flux-lower-bound}
    m_0
    \leq
    \|\sigma_\delta\|_{L^1(E_\delta,\lambda_\delta)}.
\end{equation}

\begin{theorem}[Recovery of optimal fluxes at the sparse endpoint]
\label{thm:sparse-optimal-flux-recovery}
Assume that \(\theta_\delta\to0\). For every \(\delta>0\), let
\((a_\delta,\sigma_\delta)\) be an optimal pair for the normalized
complementary problem, with \(\sigma_\delta\) chosen antisymmetric. Then
\[
    \|\sigma_\delta\|_{L^1(E_\delta,\lambda_\delta)}
    \to m_0.
\]
Moreover, the measures
\[
    \mu_\delta
    =
    R_\delta\sigma_\delta\,\mathcal{L}^d
\]
are relatively weak-star compact in
\(\mathcal{M}(\operatorname{cl}\O_{\delta_*};\Rn)\), and
every weak-star cluster point is an optimal solution of the measure-valued problem in~\eqref{eq:m0-local-flux}.
\end{theorem}
\begin{proof}
By \eqref{eq:sparse-flux-upper-bound}, \eqref{eq:sparse-value-convergence}, and
\(|\Od|\to|\O|\),
\[
    \limsup_{\delta\downarrow0}
    \|\sigma_\delta\|_{L^1(E_\delta,\lambda_\delta)}
    \leq m_0.
\]
The reverse inequality follows from~\eqref{eq:sparse-flux-lower-bound}. 
Therefore,
\[
    \|\sigma_\delta\|_{L^1(E_\delta,\lambda_\delta)}
    \to m_0.
\]
The remaining assertions follow from Proposition~\ref{prop:normalized-flux-recovery}.
\end{proof}

\begin{remark} 
For completeness, Appendix~\ref{app:uncapped-exact-localization} shows that removing the pointwise
conductivity cap produces an exact reference problem whose reduced
value equals \(m_0\) and is thus independent of the interaction horizon. 
\end{remark}

%% file: 05_discussion.tex
\section{Discussion: localization, homogenization, and local free-material optimization} \label{sec:discussion} 

The preceding results describe three distinct material/localization regimes.  
%
In the sparse regime \(\theta_\delta\to 0\), the limiting problem is the Lipschitz-constrained problem identified in Section~\ref{sec:intermediate-material}
and further characterized in Section~\ref{sec:sparse-flux-recovery}.
This is the scalar diffusion analogue of the vanishing-material limits that arise in
classical optimal design, where compliance problems under increasingly small material budgets lead to singular or measure-valued optimal design problems; see, for example, 
\cite{allaire2012shape,cherkaev2012variational,bendsoe2013topology} for the homogenization and structural optimization background. 
In the elasticity setting, the corresponding vanishing-weight limit is closely related 
to Michell-type truss problems~\cite{michell1904a}; see, for example, \cite{strang1983hencky,bendsoe1993michell,Olb18,bouchitte2008michell}. 
The nonlocal scalar version, including its measure-valued flux formulation, was 
studied in~\cite{evgrafov2021nonlocal}.

In the saturated regime \(\theta_\delta\to \theta \in [1,+\infty)\), see Subsection~\ref{subsec:saturated}, it becomes the isotropic \(p\)-Dirichlet energy. 
In the intermediate regime, however, the limiting material retains a directional internal
variable. 

Throughout the rest of this section, we consider the genuinely intermediate range \[ 0<\theta<1. \]

\subsection{Localization with a directional material variable}
The limiting upper-tail functional appearing in the intermediate regime admits the
representation 
\[ 
  \mathcal{E}_\theta(u) = \sup_{a\in\mathcal{A}_\theta} \int_\O W_a(x,\nabla u(x))\dx, 
\] 
where the local constitutive energy density is 
\[\begin{aligned}
  W_a(x,\xi) &= 
  \frac{|\O|}{p\theta} \int_{\mathbb{S}^{d-1}} a(x,s)|\xi\cdot s|^p \,\mathrm{d}\nu(s),
  \qquad\text{and}\\
  \mathcal{A}_\theta &= \left\{ a\in L^\infty( \O\times\mathbb{S}^{d-1} ): 
  0\leq a\leq1,\quad \frac{1}{|\O|} \int_\O\int_{\mathbb{S}^{d-1}} a(x,s) \,\mathrm{d}\nu(s)\,\mathrm{d}x \leq\theta \right\}. 
\end{aligned}\] 
Thus the limiting state equation is local in the spatial variable, but the admissible 
constitutive laws are parameterized by a density on \(\O\times\mathbb{S}^{d-1}\). 
In this sense, localization eliminates the bond length while retaining the bond direction.
The constraint on \(a\) is global in both position and direction. 
Consequently, the local directional mass 
\[ 
  \vartheta_a(x) = \int_{\mathbb{S}^{d-1}} a(x,s)\,\mathrm{d}\nu(s) 
\] 
may vary with \(x\), subject to 
\[ 
  0\leq\vartheta_a(x)\leq 1, \qquad 
  \frac{1}{|\O|} \int_\O\vartheta_a(x)\,\mathrm{d}x \leq\theta. 
\]
The field \(\vartheta_a\) may be viewed as a local material density, which is often used as
a design variable in optimal design problems, whereas the normalized measure 
\[ \frac{a(x,s)}{\vartheta_a(x)}\,\mathrm{d}\nu(s), \qquad \vartheta_a(x)>0, \]
describes how this material is distributed among bond directions. 
For general \(p\), the resulting constitutive law need not be represented by a 
second-order tensor.
Instead, for each fixed admissible \(a\), it is described by the convex, positively \(p\)-homogeneous energy density \(\xi\mapsto W_{a}(x,\xi)\). 
The quadratic case is therefore special in that the directional occupation can be identified with a conductivity tensor.

\subsection{The quadratic case and local free-material optimization}
Let \(p=2\). 
For \(a\in\mathcal{A}_\theta\), define the normalized conductivity tensor 
\[ 
  K_a(x) = d\int_{\mathbb{S}^{d-1}} a(x,s)s\otimes s \,\mathrm{d}\nu(s). 
\] 
The normalization factor \(d\) is chosen so that \( K_1=I, \) because 
\( \int_{\mathbb{S}^{d-1}} s\otimes s\,\mathrm{d}\nu(s) = \frac{1}{d} I. \)
The state energy associated with a fixed occupation density then takes the form 
\[ 
\int_\O W_a(x,\nabla u(x))\dx = 
  \frac{|\O|}{2d\theta} \int_\O \nabla u(x)\cdot K_a(x)\nabla u(x) \,\mathrm{d}x.
\] 
Moreover, 
\[ 
  \operatorname{tr}K_a(x) = d\int_{\mathbb{S}^{d-1}} a(x,s)\,\mathrm{d}\nu(s) 
   = d\vartheta_a(x), 
\] 
and therefore the material budget constraint becomes 
\[ \int_\O \operatorname{tr}K_a(x)\,\mathrm{d}x \leq d\theta|\O|. \] 
The intermediate limit can therefore be regarded as a local free-material optimization problem with an integral trace constraint.
It is, however, not the unrestricted trace-constrained free-material problem. 
The admissible tensors belong to the class 
\begin{equation}\label{eq:sph_mom}
\mathcal{M} = \left\{ d\int_{\mathbb{S}^{d-1}} \alpha(s)s\otimes s\,\mathrm{d}\nu(s): 
      \alpha\in L^\infty(\mathbb{S}^{d-1}), \ 0\leq \alpha\leq1 \right\}. 
\end{equation}
In particular, every \(K\in\mathcal{M}\) satisfies \( 0\preceq K \preceq I, \)\footnote{Note that the upper bound \(K\preceq I\) follows from the directional cap \(\alpha\leq1\).}
where \(\preceq\) denotes the Loewner order on symmetric matrices,
and 
\( \operatorname{tr}K = d\int_{\mathbb{S}^{d-1}} \alpha(s)\,\mathrm{d}\nu(s). \)
 
Thus the bond model permits anisotropic redistribution of the available conductivity, 
but it does not permit arbitrary concentration of the trace into a singular directional 
measure. 
By contrast, an unrestricted positive-semidefinite trace-constrained free-material 
problem admits all tensors satisfying 
\[ K\succeq0, \qquad \operatorname{tr}K=d\vartheta \] 
at a point having local material density \(\vartheta\). 
Under a single loading direction \(e\), such a class permits the rank-one tensor 
\[ K=d\vartheta\,e\otimes e. \] 
The bond-induced class is therefore a restricted local free-material class whose 
additional constraint is encoded by the bounded spherical moment representation~\eqref{eq:sph_mom}.

\subsection{Comparison with classical composite bounds} 
The bond material budget constraint in the present model should not be identified with 
the spatial volume-fraction constraint of a classical two-phase composite. 
In a void-solid conductivity problem, a fraction \(\vartheta\) of physical space is 
occupied by the conducting phase. 
Here, by contrast, \(\vartheta\) controls the occupation of position-direction space
\( \O\times\mathbb{S}^{d-1}. \) 
A material may consequently occupy only selected directions at each spatial point. 
This distinction is already visible in a homogeneous two-dimensional example. 
Let \(d=2\), fix a unit vector \(e\), and let \(\alpha_\vartheta\) be the characteristic 
function of the two angular caps centred at \(\pm e\) whose total normalized 
angular measure is \(\vartheta\).
To compute the resulting material tensor, 
we introduce the coordinates such that \(e=(1,0)\), \(e^\perp=(0,1)\), and write 
\[ 
  s(\varphi) = \cos\varphi\,e+\sin\varphi\,e^\perp, \qquad 
  \mathrm{d}\nu(\varphi)=\frac{\mathrm{d}\varphi}{2\pi}. 
\] 
Set \(\beta=\pi\vartheta/2\), and let \(\alpha_\vartheta\) be the characteristic function of 
\[ [-\beta,\beta] \cup [\pi-\beta,\pi+\beta]. \] 
This set has normalized angular measure \(\vartheta\). 
Since \(s(\varphi+\pi)\otimes s(\varphi+\pi) =s(\varphi)\otimes s(\varphi)\), we have 
\[  
  K_{\alpha_\vartheta} = 2\int_{\mathbb{S}^1} \alpha_\vartheta(s)s\otimes s\,\mathrm{d}\nu(s) 
  = \frac{2}{\pi} \int_{-\beta}^{\beta} s(\varphi)\otimes s(\varphi)\,d\varphi. 
\] 
The mixed term vanishes by symmetry, while 
\[ \begin{aligned} 
  \frac{2}{\pi} \int_{-\beta}^{\beta} \cos^2\varphi\,\mathrm{d}\varphi &= 
  \vartheta+\frac{\sin(\pi\vartheta)}{\pi}, \\ 
  \frac{2}{\pi} \int_{-\beta}^{\beta} \sin^2\varphi\,d\varphi &= 
  \vartheta-\frac{\sin(\pi\vartheta)}{\pi}. 
\end{aligned} \] 
Consequently, the associated normalized conductivity is 
\[ K_{\alpha_\vartheta} = \left( \vartheta+\frac{\sin(\pi\vartheta)}{\pi} \right)e\otimes e + \left( \vartheta-\frac{\sin(\pi\vartheta)}{\pi} \right)e^\perp\otimes e^\perp. 
\]
In particular, 
\[ e\cdot K_{\alpha_\vartheta}e = \vartheta+ \frac{\sin(\pi\vartheta)}{\pi}. \] 
For \(0<\vartheta<1\), 
\[ \vartheta < \vartheta+ \frac{\sin(\pi\vartheta)}{\pi} < 2\vartheta. \] 
The first inequality shows that the directional conductivity exceeds the Voigt bound 
in direction \(e\) for a classical mixture of unit conductor and void having spatial
conducting fraction \(\vartheta\).
This is not a violation of the classical bound, since the bond occupation fraction and 
the spatial volume fraction describe different admissible material classes. 
The second inequality compares the bond-induced tensor with unrestricted 
trace-constrained local free-material optimization. 
Indeed, 
\( \operatorname{tr}K_{\alpha_\vartheta}=2\vartheta, \)
and the unrestricted positive-semidefinite free-material optimization can place the entire
trace in direction \(e\), giving 
\(K_{\text{FMO}} = 2\vartheta\,e\otimes e. \)
For a homogeneous uniaxial experiment in direction \(e\), 
after omitting the common positive normalization factors, the corresponding scalar compliances satisfy 
\[ \frac{1}{2\vartheta} < \frac{1}{ \vartheta+\sin(\pi\vartheta)/\pi } < \frac{1}{\vartheta}. \]
For this homogeneous uniaxial experiment,
the directional bond model lies between unrestricted trace-constrained free-material optimization and classical void-solid composite design. 
The comparison does not identify the class \(\mathcal{M}\) with a classical
conductivity \(G\)-closure. 
Such an identification would require specifying constituent phases and proving that the 
spherical moment tensors are realizable as effective tensors of spatial microstructures
with the corresponding phase proportions. 
Conversely, failure of a tensor to belong to a given classical \(G\)-closure does not 
imply that it is physically inadmissible for the bond model, since the two models are 
based on different material budgets and different notions of material occupation.

\subsection{Further questions} 
The preceding observations suggest several questions that are not addressed by the 
state-energy convergence proved above. 
First, it would be useful to characterize the class \(\mathcal{M}\), and more generally 
its fixed-mass sections 
\[ 
\mathcal{M}_\vartheta = \left\{ d\int_{\mathbb{S}^{d-1}} \alpha(s)s\otimes s\,\mathrm{d}\nu(s): 0\leq \alpha\leq1,\quad \int_{\mathbb{S}^{d-1}} \alpha\,\mathrm{d}\nu=\vartheta \right\}. 
\] 
The positivity, trace, and upper-bound conditions above are immediate, but an intrinsic
characterization of all tensors admitting such a bounded spherical moment representation
would clarify the precise relation with unrestricted local free-material design.
Second, one may ask whether some or all tensors in \(\mathcal{M}_\vartheta\) can be
realized by classical composites made from prescribed constituent phases. 
This is a separate homogenization problem and depends on how a directional bond occupation
is to be translated into a spatial microstructure and its phase fractions. 
Third, Theorem~\ref{thm:intermediate-localization} concerns
optimal values and states, but it would be interesting to describe the convergence of optimal occupations or conductivities themselves. Weak-star compactness of the occupations is not sufficient for this purpose, since the coefficients and the associated strains may develop correlated oscillations. A joint material-state limit would therefore require additional homogenization or compensated-compactness arguments, which are avoided here by eliminating the occupation density before localization.
Finally, the uniform radial kernel used here produces the uniform directional measure 
\(\nu\). 
More general kernels should lead to weighted directional reference measures, and 
therefore to different moment classes and different local constitutive restrictions. 
This may provide a way to encode preferred directions already at the level of the 
underlying nonlocal interaction law.

%% file: appx_A.tex
\section{Symmetric and antisymmetric two-point fields}\label{app:two-point}
We recall the notation for symmetric and antisymmetric two-point fields.
Let \(S\subset\Rn\times\Rn\) be an open set satisfying 
\[ S=S^{\mathrm T}, \qquad S^{\mathrm T} = \left\{ (x,x'):(x',x)\in S \right\}. \] 
For a measurable two-point field \(\phi\), we write 
\[ \phi^{\mathrm T}(x,x') = \phi(x',x), \] 
and define its symmetric and antisymmetric parts by 
\[ \phi^{\mathrm s} = \frac{\phi+\phi^{\mathrm T}}{2}, \qquad 
\phi^{\mathrm a} = \frac{\phi-\phi^{\mathrm T}}{2}. 
\] 
We denote the  closed subspaces of \(L^r(S)\) containing only symmetric (resp., only antisymmetric) functions by \(L^r_{\sym}(S)\) (resp.\ \(L^r_{\asym}(S)\)).
\begin{lemma}[Symmetric and antisymmetric two-point fields] \label{lem:symmetric-antisymmetric-fields} 
  Let \(1\leq r\leq+\infty\). 
  Then the following statements hold. 
  \begin{enumerate} 
  \item The maps \( P_{\sym}\phi =\phi^{\sym}\) and \(P_{\asym}\phi =\phi^{\asym} \) are contractive projections on \(L^r(S)\). 
  In particular, \[ L^r(S) = L^r_{\sym}(S) \oplus L^r_{\asym}(S). \] 
  \item If \(1\leq r<+\infty\), then \( C_c^\infty(S)\cap L^r_{\sym}(S) \) is dense in \(L^r_{\sym}(S)\), and \( C_c^\infty(S)\cap L^r_{\asym}(S) \) is dense in \(L^r_{\asym}(S)\). 
  \item Let \(1\leq r\leq+\infty\), and let \(r'\) be its H{\"o}lder conjugate. 
  If \( \phi\in L^r_{\asym}(S)\), \(\psi\in L^{r'}_{\sym}(S), \) then 
  \( \int_S\phi\psi=0. \) 
  \item If \(1\leq r<+\infty\) and \(\phi\in L^{r'}_{\mathrm a}(S)\), then 
  \[ \|\phi\|_{L^{r'}(S)} = \sup_{\substack{ \psi\in L^r_{\mathrm a}(S)\\ \|\psi\|_{L^r(S)}\leq1 }} \int_S\phi\psi. \] 
  An analogous identity holds with the symmetric and antisymmetric spaces interchanged. 
  \item Let \(1\leq r<+\infty\), and let \(r'\) be its H{\"o}lder conjugate. Suppose that \(\phi\in L^{r'}(S)\). If \[ \int_S\phi\psi=0 \qquad \text{for every } \psi\in C_c^\infty(S)\cap L^r_{\asym}(S), \] then \(\phi\in L^{r'}_{\sym}(S). \)
  An analogous identity holds with the symmetric and antisymmetric spaces interchanged. 
  \end{enumerate}
\end{lemma} 
\begin{proof}
  1. The transposition map \( R\phi =\phi^{\mathrm T} \) is a linear isometry 
  of \(L^r(S)\), because \(S=S^{\mathrm T}\).
  Consequently, 
  \[ \begin{aligned} 
    \|P_{\sym}\phi\|_{L^r(S)} &\leq \frac12 \left( \|\phi\|_{L^r(S)} + \|\phi^{\mathrm T}\|_{L^r(S)} \right) = \|\phi\|_{L^r(S)}, \\ 
    \|P_{\asym}\phi\|_{L^r(S)} &\leq \frac12 \left( \|\phi\|_{L^r(S)} + \|\phi^{\mathrm T}\|_{L^r(S)} \right) = \|\phi\|_{L^r(S)}. 
  \end{aligned} \]
  The identities \[ \phi=P_{\mathrm s}\phi+P_{\mathrm a}\phi, \qquad L^r_{\mathrm s}(S)\cap L^r_{\mathrm a}(S)=\{0\}, \] prove the first assertion. 
  
  \smallskip\noindent
  2. Let now \(1\leq r<+\infty\), and suppose that \(\phi\in L^r_{\asym}(S)\). Choose \(\phi_n\in C_c^\infty(S)\) such that 
  \( \phi_n\to\phi \) in \(L^r(S). \) 
  Since \(S\) is symmetric, 
  \[ P_{\asym}\phi_n = \frac{\phi_n-\phi_n^{\mathrm T}}2 \in C_c^\infty(S)\cap L^r_{\asym}(S). \] 
  Contractivity of \(P_{\asym}\) gives 
  \[ \|P_{\asym}\phi_n-\phi\|_{L^r(S)} = \|P_{\asym}(\phi_n-\phi)\|_{L^r(S)} \leq \|\phi_n-\phi\|_{L^r(S)} \to 0. \] 
  The symmetric case is identical. 
  
  \smallskip\noindent
  3. If \(\phi\) is antisymmetric and \(\psi\) is symmetric, invariance of \(S\) under endpoint interchange gives \[ \int_S\phi\psi = \int_S\phi^{\mathrm T}\psi^{\mathrm T} = -\int_S\phi\psi, \] and hence the integral vanishes. 
  
  \smallskip\noindent
  4. Let \(\phi\in L^{r'}_{\asym}(S)\). 
  By the usual \(L^{r'}\)-duality formula and the preceding orthogonality, 
  \[ \begin{aligned} 
    \|\phi\|_{L^{r'}(S)} &= \sup_{\substack{ \psi\in L^r(S)\\ \|\psi\|_{L^r(S)}\leq1 }} \int_S\phi\psi = \sup_{\substack{ \psi\in L^r(S)\\ \|\psi\|_{L^r(S)}\leq1 }} \int_S\phi\,P_{\asym}\psi \leq \sup_{\substack{ \eta\in L^r_{\asym}(S)\\ \|\eta\|_{L^r(S)}\leq1 }} \int_S\phi\eta \leq \|\phi\|_{L^{r'}(S)}. 
  \end{aligned} \] 
  This proves the asserted equality. The symmetric case is analogous.

  \smallskip\noindent
  5. For the final assertion, suppose that \(\phi\) annihilates \(C_c^\infty(S)\cap L^r_{\asym}(S)\). 
  By the density statement, 
  \[ \int_S\phi\psi=0 \qquad \text{for every }\psi\in L^r_{\asym}(S). \] 
  Since 
  \[ \int_S\phi\psi = \int_S P_{\asym}\phi\,\psi \qquad \text{for every }\psi\in L^r_{\asym}(S), \] 
  the restricted duality formula gives 
  \( \|P_{\asym}\phi\|_{L^{r'}(S)}=0. \)
  Thus \(P_{\asym}\phi=0\), and \(\phi\in L^{r'}_{\sym}(S)\). 
  The other implication is analogous.
\end{proof}

%% file: appx_B.tex
\section{Sparse local optimal design duality and measure relaxation}
\label{appx:loc_rlx}

In this appendix we establish the duality relations between the Lipschitz-constrained state problem~\eqref{eq:m0-local-state}, the \(L^1\)-flux problem, and its measure-valued relaxation~\eqref{eq:m0-local-flux}. 
The argument uses two perturbation spaces because the relaxation and the final state space require different dual topologies.

We first allow the divergence datum to vary in \(\mathcal{D}'(\O)\). 
This topology is natural for identifying the weak-star relaxation of the \(L^1\)-flux functional and yields a dual problem over \(C_c^\infty(\O)\). 
We then restrict the perturbation variable to \(L^1(\O)\). 
The corresponding dual space is \(L^\infty(\O)\), and finiteness of the conjugate 
identifies its effective domain with \(\operatorname{Lip}_0(\O)\). 
The second perturbation therefore recovers the precise state space appearing in~\eqref{eq:m0-local-state}, while the first identifies the measure-valued relaxation 
in~\eqref{eq:m0-local-flux}.

To this end, let \(\O\subset\Rn\) be a bounded Lipschitz domain, and let \(f \in \mathcal{D}'(\O)\) be a given distribution.
We define the following convex functional \(J_f:\mathcal{M}(\operatorname{cl}\O;\R^d)\to \overline{\R}\):
\begin{equation*}
    \begin{aligned}
    J_f(\mu) &= \begin{cases}
        \lVert \sigma\rVert_{L^1(\O;\Rn)}, \quad&\text{if\ } \mu = \sigma\mathcal{L}^d, \sigma \in L^1(\O;\Rn), \text{\ and\ } \operatorname{div}\sigma = f, \;\text{in\ } \mathcal{D}'(\O),\\
        +\infty,\quad&\text{otherwise}.
    \end{cases}
\end{aligned}
\end{equation*}
Note that we enforce the constraint \(\operatorname{div}\mu = f\) distributionally, that is,
\begin{equation*}
    -\int_{\operatorname{cl}\O} \nabla\phi(x)\cdot\mathrm{d}\mu(x) = \langle f, \phi \rangle,\qquad \forall \phi \in \mathcal{D}(\O).
\end{equation*}
In particular, whenever \(\sigma \in W^1(\operatorname{div};\O)\), we identify \(\operatorname{div}\sigma\) and \(\operatorname{div}[\sigma\mathcal{L}^d]\).
We begin by recalling that the convex functional defined above is proper for each \(f\in L^1(\O)\).
\begin{proposition}\label{claim:distr_diver_surj}
    The distributional divergence operator \(\Div: D(\Div)=W^1(\operatorname{div}; \O)\subset L^1(\O;\Rn)\to L^1(\O)\) is surjective
    and admits a bounded right inverse.
\end{proposition}
\begin{proof}
    Let us fix an arbitrary \(f\in L^1(\O)\).
    Let \((a,b) \times Q' \subset \R\times \R^{d-1}\) be an axis aligned rectangular box containing \(\O\), and let \(f_0\in L^1((a,b) \times Q')\) be an extension of \(f\) by zero outside of \(\O\).
    Let \(\sigma_1(x_1,x_2,\dots,x_d) = \int_{a}^{x_1} f_0(t,x_2,\dots,x_d)\,\mathrm{d}t\), and let further
    \(\sigma(x) = (\sigma_1(x), 0, \dots, 0)\).
    Note that 
    \[|\sigma(x_1,x_2,\dots,x_d)|\leq \int_a^b|f_0(t,x_2,\dots,x_d)|\,\mathrm{d}t,\] 
    and consequently, owing to Fubini's theorem, \(\sigma \in L^1((a,b)\times Q';\Rn)\) with
    \(\lVert \sigma\rVert_{L^1(\O;\Rn)} \leq \lVert \sigma\rVert_{L^1((a,b)\times Q';\Rn)} \leq |b-a| \lVert f_0\rVert_{L^1((a,b)\times Q')} = |b-a| \lVert f\rVert_{L^1(\O)}\).
    Furthermore, for each \(\phi \in C^\infty_c(\O)\) we have
    \begin{equation*}
    -\langle \sigma, \nabla \phi \rangle_{\O}
    =-\langle \sigma_1, \partial_{x_1}\phi \rangle_{\O}
    =\langle f_0, \phi \rangle_{\O}
    =\langle f, \phi \rangle_{\O},
    \end{equation*}
    where we have used Fubini's theorem and the fundamental theorem of calculus in 1D. Consequently, \(\sigma \in W^1(\operatorname{div};\O)\) with \(\operatorname{div}\sigma = f\) in \(\mathcal{D}'(\O)\).
    The mapping \( f\mapsto \sigma|_\O \) is linear, and the preceding estimate shows that it defines a bounded right inverse of the distributional divergence.
\end{proof}

Let us now introduce a proper convex perturbation functional
\(\Phi_f: \mathcal{M}(\closure\O;\Rn) \times \mathcal{D}'(\O) \to \overline{\R}\),
\(\Phi_f(\mu,g) = J_{f+g}(\mu)\).
Its convex conjugate \(\Phi_f^*: C(\closure \O;\Rn) \times C^\infty_c(\O) \to \overline{\R}\)
is given by
\begin{equation*}
\begin{aligned}
\Phi_f^*(v,u) &= \sup_{\substack{%
\mu \in \mathcal{M}(\closure\O;\Rn)\\
g \in \mathcal{D}'(\O)}}[\langle \mu, v \rangle + \langle g, u\rangle
- J_{f+g}(\mu)]
\\&= \sup_{\substack{%
\sigma \in L^1(\O;\Rn)\\
g =\Div\sigma - f}}[\langle v - \nabla u, \sigma\rangle_\O
- \lVert \sigma\rVert_{L^1(\O)} - \langle f, u\rangle]
\\&=\begin{cases}
-\langle f, u\rangle, &\quad \text{if\ } \lVert v-\nabla u\rVert_{L^\infty(\O;\Rn)}\leq 1,\\
+\infty, &\quad\text{otherwise},
\end{cases}
\end{aligned}
\end{equation*}
and the biconjugate \(\Phi_f^{**}: \mathcal{M}(\closure\O;\Rn) \times \mathcal{D}'(\O) \to \overline{\R}\) by
\begin{equation*}
\begin{aligned}
\Phi^{**}_f(\mu,g) &= \sup_{\substack{%
v \in C(\closure\O;\Rn)\\
u \in C^\infty_c(\O)}}[\langle \mu, v \rangle + \langle g, u\rangle
- \Phi_f^*(v,u)]
= \sup_{\substack{%
v \in C(\closure\O;\Rn)\\
u \in C^\infty_c(\O)\\
\lVert v-\nabla u\rVert_{L^\infty(\O;\Rn)}\leq 1}}
[\langle \mu, v \rangle + \langle f+g, u\rangle]
\\&= \sup_{\substack{%
w=v-\nabla u \in C(\closure\O;\Rn)\\
u \in C^\infty_c(\O)\\
\lVert w\rVert_{L^\infty(\O;\Rn)}\leq 1}}
[\langle \mu, w + \nabla u\rangle + \langle f+g, u\rangle]
\\&=\begin{cases}
        \lVert\mu\rVert_{\mathcal{M}(\operatorname{cl}\O;\Rn)}, \quad&\text{if\ } \operatorname{div}\mu = f+g, \;\text{in\ } \mathcal{D}'(\O),\\
         +\infty,\quad&\text{otherwise}
\end{cases}
\\&=: H_{f+g}(\mu).
\end{aligned}
\end{equation*}

Thus for each \(f\in \mathcal{D}'(\O)\) the problems \(\inf_{\mu\in \mathcal{M}(\closure\O;\Rn)} H_f(\mu) = \inf_{\mu\in \mathcal{M}(\closure\O;\Rn)} \Phi^{**}_f(\mu,0)\)
and \(\sup_{u\in C^\infty_c(\O),\lVert \nabla u\rVert_{L^\infty(\O;\Rn)}\leq 1} \langle f, u\rangle = 
\sup_{u\in C^\infty_c(\O)} [-\Phi_f^*(0,u)]\) are related via duality~\cite[Chapter III]{ekeland1976convex}.

\begin{proposition}\label{prop:normal}
    For each \(f\in L^1(\O)\) the optimal value function \(h(v) = \inf_{u\in C^\infty_c(\O)}[\Phi_f^*(v,u)]\) is finite and l.s.c.\ (in fact, continuous) at \(v=0\).
\end{proposition}
\begin{proof}
    For \(c\geq 0\) let \(U_c = \{\, u\in C^\infty_c(\O) \colon \lVert \nabla u\rVert_{L^\infty(\O;\Rn)} \leq c \,\}\), and consider the function
    \[
        \tilde{h}(c) = \inf_{u\in U_c} \langle -f, u\rangle_\O.
    \]
    Since \(0 \in U_c\), \(\forall c\geq 0\), we have that \(\tilde{h}(c)\leq 0\).
    On the other hand, for each \(u\in U_c\) we can write
    \[
        -\langle f, u\rangle_\O 
        \geq - \lVert u\rVert_{L^\infty(\O)} \lVert f\rVert_{L^1(\O)}
        \geq - c\operatorname{diam}(\O)\lVert f\rVert_{L^1(\O)},
    \]
    owing to H{\"o}lder's inequality and Lipschitz continuity of \(u\).
    Consequently,
    \(
        \tilde{h}(c) \geq  - c\operatorname{diam}(\O)\lVert f\rVert_{L^1(\O)}.
    \)
    Finally, \(\tilde{h}:[0,+\infty) \to \R\) is degree one homogeneous, since \(U_c = c U_1\), and
    the objective function \(u\mapsto -\langle f, u\rangle_{\O}\) is linear.
    Thus \(\tilde{h}\) is continuous at any \(c>0\).

    With this in mind, for \(\|v\|_{L^\infty(\O;\Rn)}<1\), we can write the estimate
    \[
        \tilde{h}(1+\lVert v\rVert_{L^\infty(\O;\Rn)}) \leq h(v) \leq \tilde{h}(1-\lVert v\rVert_{L^\infty(\O;\Rn)}),
    \]
    thereby concluding the proof.
\end{proof}

\begin{corollary}
    For each \(f\in L^1(\O)\)  we have the equality
    \begin{equation}\label{eq:no_gap}
    \inf_{\mu\in \mathcal{M}(\closure\O;\Rn)} H_f(\mu) = \sup_{\substack{u\in C^\infty_c(\O)\\\lVert \nabla u\rVert_{L^\infty(\O;\Rn)}\leq 1}} \langle f, u\rangle,
    \end{equation}
    that is, there is no duality gap.
\end{corollary}
\begin{proof}
    This is equivalent to the conclusion of Proposition~\ref{prop:normal}; see~\cite[Proposition~2.1 in Chapter III]{ekeland1976convex}.
\end{proof}
Note that the problem on the right hand side of~\eqref{eq:no_gap} need not admit a solution.
However, we can safely replace \(C^\infty_c(\O)\) with a larger space \(\operatorname{Lip}_0(\O)\) without altering the supremum's value.
\begin{proposition}
    For each \(f\in L^1(\O)\)  we have the equality
    \begin{equation*}
    \sup_{\substack{u\in C^\infty_c(\O)\\\lVert \nabla u\rVert_{L^\infty(\O;\Rn)}\leq 1}} \langle f, u\rangle_\O =
    \sup_{\substack{u\in \operatorname{Lip}_0(\O)\\\lVert \nabla u\rVert_{L^\infty(\O;\Rn)}\leq 1}} \langle f, u\rangle_\O,
    \end{equation*}
    and the supremum on the right is attained.
\end{proposition}
\begin{proof}
    The equality is the direct consequence of two facts: (i) any function \(u\in \operatorname{Lip}_0(\O)\) has a Lipschitz continuous representative with Lipschitz constant bounded by \(\lVert \nabla u\rVert_{L^\infty(\O;\Rn)}\);
    (ii) any Lipschitz continuous function \(u\) such that \(u|_{\partial \O}=0\) can be uniformly approximated by \(C^\infty_c(\O)\) functions having the same Lipschitz constant, see~\cite[Remark~4.1]{deville2019approximation}.

    Existence of solutions is easily obtained using the direct method of calculus of variations.
    Indeed, let us take an arbitrary bounded open set \(U\supset \closure \O\).  Each maximizing sequence is bounded in \(W^{1,\infty}(U)\) as \(\lVert u_k\rVert_{L^\infty(U)} \leq \operatorname{diam}(\O)\lVert \nabla u_k\rVert_{L^\infty(\O;\Rn)}\leq \operatorname{diam}(\O)\). Therefore, it contains a subsequence \((u_{k'})\) with \(u_{k'}\to \bar{u}\), uniformly in \(C(\closure U)\) (see for example \cite[Theorem 9.16]{brezis2010functional}) and \(\nabla u_{k'}\wsto \bar{v}\) in \(L^\infty(U;\Rn)\).
    One checks easily that \(\bar{v} = \nabla \bar{u}\), \(\lVert \bar{v}\rVert_{L^\infty(\O;\Rn)} \leq 1\), and obviously \(\langle f, \bar{u}\rangle_\O = \lim_{k'\to\infty} \langle f, u_{k'}\rangle_\O\) and \(\bar{u}|_{U\setminus \O} = 0\).  Thus any such \(\bar{u}\in \operatorname{Lip}_0(\O)\)  and is an optimal solution.
\end{proof}

The preceding perturbation identifies the relaxed measure problem, but its dual variable belongs to \(\mathcal{D}(\O)=C_c^\infty(\O)\). 
To obtain the dual problem directly in \(\operatorname{Lip}_0(\O)\), we now restrict the perturbation variable to \(L^1(\O)\).
To this end we define \(\check{\Phi}_f: \mathcal{M}(\closure\O;\Rn) \times L^1(\O) \to \overline{\R}\) be a restriction of the perturbation functional \(\Phi_f\), that is, \(\check{\Phi}_f(\mu,\cdot) = \Phi_f(\mu,\cdot)|_{L^1(\O)}\), \(\forall \mu \in \mathcal{M}(\closure\O;\Rn)\).
Thus the primal problem \(\inf_{\mu \in \mathcal{M}(\closure\O;\Rn)} \check{\Phi}_f(\mu,0) = \inf_{\mu \in \mathcal{M}(\closure\O;\Rn)} J_f(\mu)\) remains unchanged, but the value function and the dual change a little.

\begin{proposition}\label{prop:normal2}
   For each \(f\in L^1(\O)\), the optimal value function \(i(g) = \inf_{\mu \in \mathcal{M}(\closure\O;\Rn)} \check{\Phi}_f(\mu,g)\) is finite and l.s.c.\ at \(g=0\).
\end{proposition}
\begin{proof}
   Let \(\Div^{-1}: L^1(\O) \to L^1(\O;\Rn)\) be a bounded right inverse of \(\Div: D(\Div)\subset L^1(\O;\Rn)\to L^1(\O)\), 
   for example the one constructed in Proposition~\ref{claim:distr_diver_surj}.
   Then \(\forall g\in L^1(\O)\) we have the inequalities \(0\leq i(g) \leq \lVert \Div^{-1}(f+g)\rVert_{L^1(\O;\Rn)} < +\infty\), and in particular \(i(0)<+\infty\).

   Let us now fix an arbitrary \(\epsilon >0\).
   Then for each \(g \in L^1(\O)\) there exists \(\sigma_{g,\epsilon} \in L^1(\O;\Rn)\) such that \(\Div\sigma_{g,\epsilon}=f+g\) and 
   \(\lVert \sigma_{g,\epsilon}\rVert_{L^1(\O;\Rn)}=\check{\Phi}_f(\sigma_{g,\epsilon},g)< i(g)+\epsilon\).
   Therefore,
   \(i(0) \leq \check{\Phi}_f(\sigma_{g,\epsilon}- \Div^{-1}(g),0) = \lVert \sigma_{g,\epsilon} - \Div^{-1}(g)\rVert_{L^1(\O;\Rn)}  \leq 
   \lVert \sigma_{g,\epsilon}\rVert_{L^1(\O;\Rn)} + \lVert \Div^{-1}(g)\rVert_{L^1(\O;\Rn)} < i(g)+\epsilon + \lVert \Div^{-1}(g)\rVert_{L^1(\O;\Rn)}\).
   By letting \(g\to 0\) in \(L^1(\O)\) we obtain the inequality \(i(0)\leq \liminf_{g\to 0} i(g) + \epsilon\), which is precisely the lower
   semicontinuity given the fact that \(\epsilon >0\) is arbitrary.
\end{proof}

\begin{proposition}
    For each \(f\in L^1(\O)\)  we have the equality
    \begin{equation}\label{eq:no_gap2}
    \inf_{\mu \in \mathcal{M}(\closure\O;\Rn)} J_f(\mu)=
    \sup_{\substack{u\in \operatorname{Lip}_0(\O)\\\lVert \nabla u\rVert_{L^\infty(\O;\Rn)}\leq 1}} \langle f, u\rangle_\O.
    \end{equation}
\end{proposition}
\begin{proof}
    Note that the left-hand side of~\eqref{eq:no_gap2} is \(\inf_{\mu \in \mathcal{M}(\closure\O;\Rn)} \check{\Phi}_f(\mu,0)\), which in view of Proposition~\ref{prop:normal2}
    and \cite[Proposition~2.1 in Chapter III]{ekeland1976convex} equals \(\sup_{u\in L^\infty(\O)} [-\check{\Phi}_f^*(0,u)]\).
    We will show that the latter expression is precisely the right-hand side of~\eqref{eq:no_gap2}.

    Let \((v,u)\in C(\closure \O; \Rn)\times L^\infty(\O)\) be arbitrary.
    We compute
    \begin{equation*}
        \begin{aligned}
            \check{\Phi}_f^*(v,u) &= \sup_{\substack{\mu \in \mathcal{M}(\closure\O;\Rn)\\
    g \in L^1(\O)}} [\langle \mu, v\rangle + \langle g, u\rangle_\O - \check{\Phi}_f(\mu,g)]
    = \sup_{\substack{\sigma \in W^1(\Div;\O)\\
    g=\Div \sigma-f}} [\langle \sigma, v\rangle_\O + \langle g, u\rangle_\O - \check{\Phi}_f(\sigma,g)]
    \\&=\sup_{\sigma \in W^1(\Div;\O)}
    [\langle \sigma, v\rangle_\O - \lVert \sigma\rVert_{L^1(\O;\Rn)} + \langle \Div \sigma, u\rangle_\O] - \langle f, u\rangle_\O
    \\&=\sup_{\sigma \in C^\infty_c(\Rn;\Rn)}
    [\langle \sigma, v\rangle_\O - \lVert \sigma\rVert_{L^1(\O;\Rn)} + \langle \Div \sigma, u\rangle_\O] - \langle f, u\rangle_\O,
        \end{aligned}
    \end{equation*}
    where the last equality is owing to the density of restriction of \(C^\infty_c(\R^d;\R^d)\) to \(\O\) in \(W^1(\Div;\O)\), see~\cite[Theorem III.2.1]{galdi2011introduction}.
    If \(\exists \sigma_1, \sigma_2, \ldots \in C^\infty_c(\R^d;\R^d)\), such that
    \(\langle \Div\sigma_k, u \rangle_\O > k\lVert \sigma_k\rVert_{L^1(\O;\R^d)}\), then 
    the supremum equals \(+\infty\),
    as \(|\langle \sigma, v\rangle_{\O}|
    \leq \lVert v\rVert_{C(\operatorname{cl}\O;\R^d)}\lVert \sigma\rVert_{L^1(\O;\R^d)}\), \(\forall \sigma\in L^1(\O;\R^d)\).
    Assuming that the supremum is finite, \(\exists C\geq 0\) such that \(\langle \Div\sigma, u \rangle_\O \leq C\lVert \sigma\rVert_{L^1(\O;\R^d)}\), \(\forall \sigma \in C^\infty_c(\R^d;\R^d)\), and consequently \(u \in \operatorname{Lip}_0(\O)\), see~\cite[Proposition 9.18]{brezis2010functional}.\footnote{Note: \cite[Proposition 9.18]{brezis2010functional} assumes \(1<p<+\infty\), but the implication \((ii)\implies(iii)\), which we need here, only relies on~\cite[Proposition 9.3]{brezis2010functional}, which remains valid for \(p=\infty\).}
    After choosing the Lipschitz representative of \(u\) \cite[Proposition~9.3]{brezis2010functional} and integrating by parts~\cite[Corollary 9.66]{leoni2017first}, we can continue the computation in this case:
    \[
    \begin{aligned}
        \check{\Phi}_f^*(v,u)|_{u\in\operatorname{Lip}_0(\O)} &=
        \sup_{\sigma \in C^\infty_c(\Rn;\Rn)}[
        \langle v-\nabla u, \sigma \rangle_{\O} - \lVert \sigma\rVert_{L^1(\O;\R^d)}
        + \underbrace{\langle \sigma\cdot \nu, u \rangle_{\partial\O}}_{=0}] - \langle f, u\rangle_\O
        \\&=
        \sup_{\sigma\in L^1(\O;\R^d)}
            (\lVert v-\nabla u\rVert_{L^\infty(\O;\R^d)} - 1)\lVert \sigma\rVert_{L^1(\O;\R^d)} - \langle f, u\rangle_\O
        \\&=
        \begin{cases}
        - \langle f, u\rangle_\O, &\qquad\text{if\ } \lVert v-\nabla u\rVert_{L^\infty(\O;\R^d)}\leq 1,\\
        +\infty,&\qquad\text{otherwise},
        \end{cases}
    \end{aligned}
    \]
    where \(\nu\) is the outward unit normal for \(\O\), and we used H{\"o}lder's inequality and the density of \(C^\infty_c(\Rn;\Rn)\) in \(L^1(\O;\Rn)\).
\end{proof}

%% file: appx_C.tex
\section{Exact horizon-independence of the nonlocal conductivity-uncapped problem}
\label{app:uncapped-exact-localization}

This appendix considers the reference problem obtained by removing the
pointwise upper bound on the bond conductivity while retaining the global
material constraint. 
We first eliminate the conductivity, then construct the \(L^1\)-graph closure of the
normalized divergence, and finally prove that the resulting problem is independent of the
interaction horizon.

Recall that the normalized occupation variable is
\(a=\kappa/\overline\kappa\). Thus the physical constraint
\(0\leq\kappa\leq\overline\kappa\) becomes \(0\leq a\leq1\); see
Subsection~\ref{subsec:rescaling}. Removing the conductivity cap amounts to
dropping the condition \(a\leq1\), while retaining the material constraint
\(\int_{E_\delta}a\,\mathrm d\lambda_\delta\leq\vartheta\).

\subsection{Elimination of the uncapped conductivity}
\label{subsec:uncapped-elimination}

For \(\vartheta>0\), define
\[
    \mathcal{A}_{\delta,\vartheta}^{\infty}
    =
    \left\{
        a\in L^1(E_\delta,\lambda_\delta):
        a\geq0,\quad
        \int_{E_\delta}a\,\mathrm d\lambda_\delta\leq\vartheta
    \right\},
    \qquad
    \mathcal{A}_{\delta,\vartheta}^{\infty,\sym}
    =
    \mathcal{A}_{\delta,\vartheta}^{\infty}
    \cap L^1_{\sym}(E_\delta,\lambda_\delta).
\]
The normalized conductivity-uncapped value is
\begin{equation*}
    i_\delta^\infty(\vartheta)
    =
    \inf_{
        (a,\sigma)\in\mathcal{A}_{\delta,\vartheta}^{\infty,\sym}\times Y_\delta^q(f)}
    \frac1q
    \left(\frac{|\Od|^2}{\vartheta}\right)^{1-q}
    \int_{E_\delta}a^{1-q}|\sigma|^q
    \,\mathrm{d}\lambda_\delta.
\end{equation*}
As before, the convention at \(a=0\) is the one introduced after
\eqref{eq:energies}.

\begin{proposition}[Elimination of the uncapped conductivity]
\label{prop:uncapped-conductivity-elimination}
Let \(\sigma\in L^1_{\asym}(E_\delta,\lambda_\delta)\). Then
\begin{equation}
\label{eq:uncapped-coefficient-elimination}
    \inf_{a\in\mathcal{A}_{\delta,\vartheta}^{\infty,\sym}}
    \frac1q
    \left(\frac{|\Od|^2}{\vartheta}\right)^{1-q}
    \int_{E_\delta}a^{1-q}|\sigma|^q
    \,\mathrm{d}\lambda_\delta
    =
    \frac{\|\sigma\|_{L^1(E_\delta,\lambda_\delta)}^q}
         {q|\Od|^{2(q-1)}}.
\end{equation}
If \(\sigma\neq0\), the minimum is attained by
\[
    a_{\delta,\sigma}
    =
    \frac{\vartheta}
         {\|\sigma\|_{L^1(E_\delta,\lambda_\delta)}}|\sigma|.
\]
For \(\sigma=0\), both sides of
\eqref{eq:uncapped-coefficient-elimination} vanish.
\end{proposition}
\begin{proof}
For every admissible \(a\), H{\"o}lder's inequality gives
\[
    \|\sigma\|_{L^1(E_\delta,\lambda_\delta)}
    \leq
    \left(
        \int_{E_\delta}a^{1-q}|\sigma|^q
        \,\mathrm d\lambda_\delta
    \right)^{1/q}
    \left(
        \int_{E_\delta}a\,\mathrm d\lambda_\delta
    \right)^{1/p}.
\]
Hence
\[
    \int_{E_\delta}a^{1-q}|\sigma|^q
    \,\mathrm d\lambda_\delta
    \geq
    \vartheta^{1-q}
    \|\sigma\|_{L^1(E_\delta,\lambda_\delta)}^q.
\]
For \(\sigma\neq0\), equality is attained by
\(a_{\delta,\sigma}\). This occupation is admissible and symmetric because
\(\sigma\) is antisymmetric. Substitution proves
\eqref{eq:uncapped-coefficient-elimination}. The case \(\sigma=0\) is
immediate.
\end{proof}

\subsection{The closed \texorpdfstring{$L^1$}{L1}-divergence}
\label{subsec:uncapped-closed-divergence}

\begin{proposition}[Closure and prescribed-divergence density]
\label{prop:density-fixed-divergence}
The operator
\[
    \operatorname{div}_\delta:
    Y_\delta^q\subset L^1(E_\delta,\lambda_\delta)
    \to L^q(\O)
\]
is closable. We continue to denote its graph closure by \(\operatorname{div}_\delta\), and denote its domain by \(Y_\delta^1\), equipped with the norm
\(
    \|\sigma\|_{Y_\delta^1}
    =
    \|\sigma\|_{L^1(E_\delta,\lambda_\delta)}
    +
    \|\operatorname{div}_\delta\sigma\|_{L^q(\O)}.
\)
For every \(f\in L^q(\O)\), the space \(Y_\delta^q(f)\) is dense in
\(Y_\delta^1(f)\) with respect to this norm. The same holds in the
antisymmetric subspaces:
\(
    Y_{\delta,\asym}^q(f)\)
    is dense in
    \(Y_{\delta,\asym}^1(f).
\)
\end{proposition}

\begin{proof}
Suppose that \(\sigma_n\to0\) in
\(L^1(E_\delta,\lambda_\delta)\) and
\(\operatorname{div}_\delta\sigma_n\to g\) in \(L^q(\O)\). For every
\(v\in\operatorname{Lip}_0(\O)\),
\[
    \langle g,v\rangle_\O
    =
    -\lim_{n\to\infty}
    \int_{E_\delta}\sigma_nD_\delta v
    \,\mathrm d\lambda_\delta
    =0,
\]
because \(\|D_\delta v\|_{L^\infty(E_\delta)}\leq
\operatorname{Lip}(v)\). Since
\(C_c^\infty(\O)\subset\operatorname{Lip}_0(\O)\), we have \(g=0\), and
the operator is closable.

Let \(\sigma\in Y_\delta^1(f)\). By definition of the graph closure, there
are \(\sigma_n\in Y_\delta^q\) such that \(\sigma_n\to\sigma\) in
\(L^1(E_\delta,\lambda_\delta)\) and
\(\operatorname{div}_\delta\sigma_n\to f\) in \(L^q(\O)\). The map
\(\operatorname{div}_\delta:Y_\delta^q\to L^q(\O)\) is surjective.
Therefore, by the open mapping theorem, there is \(C_\delta>0\) such that
for every \(g\in L^q(\O)\) one can choose \(\eta_g\in Y_\delta^q\) with
\[
    \operatorname{div}_\delta\eta_g=g,
    \qquad
    \|\eta_g\|_{Y_\delta^q}
    \leq C_\delta\|g\|_{L^q(\O)}.
\]
Choose \(\eta_n\) with
\(\operatorname{div}_\delta\eta_n
=f-\operatorname{div}_\delta\sigma_n\), and set
\(\widehat\sigma_n=\sigma_n+\eta_n\). Then
\(\widehat\sigma_n\in Y_\delta^q(f)\). Since \(\lambda_\delta\) is a
probability measure,
\[
    \|\widehat\sigma_n-\sigma\|_{L^1(E_\delta,\lambda_\delta)}
    \leq
    \|\sigma_n-\sigma\|_{L^1(E_\delta,\lambda_\delta)}
    +
    \|\eta_n\|_{L^q(E_\delta,\lambda_\delta)}
    \to0.
\]
Their divergences all equal \(f\), so convergence also holds in the
\(Y_\delta^1\)-norm.

For the antisymmetric assertion, note that
\(\operatorname{div}_\delta:Y_{\delta,\asym}^q\to L^q(\O)\) is
surjective. Indeed, if \(\operatorname{div}_\delta\eta=g\), then
antisymmetrization is contractive and
\(\operatorname{div}_\delta\eta^{\asym}=g\); see
Appendix~\ref{app:two-point}. The preceding construction may therefore be
carried out in the antisymmetric subspaces.
\end{proof}

We put
\[
    m_\delta
    =
    \inf_{\sigma\in Y_{\delta,\asym}^1(f)}
    \|\sigma\|_{L^1(E_\delta,\lambda_\delta)}.
\]
Owing to Proposition~\ref{prop:density-fixed-divergence},
\begin{equation}
\label{eq:mdelta-two-domains}
    m_\delta
    =
    \inf_{\sigma\in Y_{\delta,\asym}^q(f)}
    \|\sigma\|_{L^1(E_\delta,\lambda_\delta)}.
\end{equation}
Antisymmetrization preserves the divergence and does not increase the
complementary energy. 
Proposition~\ref{prop:uncapped-conductivity-elimination}
and~\eqref{eq:mdelta-two-domains} therefore give
\begin{equation*}
    i_\delta^\infty(\vartheta)
    =
    \frac{m_\delta^q}{q|\Od|^{2(q-1)}}.
\end{equation*}
Thus the material fraction \(\vartheta\) disappears after the uncapped
conductivity has been eliminated. 
The remaining goal of this section is to relate
\(m_\delta\) to \(m_0\) from \eqref{eq:m0-local-state}.

\subsection{Dual representation of the conductivity-uncapped value}
\label{subsec:uncapped-dual}

We use the symmetry properties collected in
Appendix~\ref{app:two-point}. In particular, \(D_\delta u\) is
antisymmetric, antisymmetrization is contractive, and smooth
antisymmetric fields are dense in the corresponding bond spaces.

\begin{lemma}[Characterization of bounded nonlocal gradients]
\label{lem:bounded-nonlocal-gradient}
Let \(u\in L^p(\O)\). Then \(D_\delta u\in L^\infty(E_\delta)\) if and
only if there is \(C\geq0\) such that
\[
    |\langle u,\operatorname{div}_\delta\sigma\rangle_\O|
    \leq
    C\|\sigma\|_{L^1(E_\delta,\lambda_\delta)}
\]
for every \(\sigma\in Y_{\delta,\asym}^1\). In this case,
\begin{equation}
\label{eq:closed-L1-integration-by-parts}
    \langle u,\operatorname{div}_\delta\sigma\rangle_\O
    =
    -\int_{E_\delta}\sigma D_\delta u
    \,\mathrm d\lambda_\delta.
\end{equation}
\end{lemma}

\begin{proof}
Assume first that \(D_\delta u\in L^\infty(E_\delta)\), and let
\(\sigma\in Y_{\delta,\asym}^1\). By definition of the graph closure,
there are \(\sigma_n\in Y_{\delta,\asym}^q\) such that
\(\sigma_n\to\sigma\) in \(L^1(E_\delta,\lambda_\delta)\) and
\(\operatorname{div}_\delta\sigma_n\to
\operatorname{div}_\delta\sigma\) in \(L^q(\O)\). Passing to the limit in
the original integration-by-parts identity gives
\eqref{eq:closed-L1-integration-by-parts}, and therefore
\[
    |\langle u,\operatorname{div}_\delta\sigma\rangle_\O|
    \leq
    \|D_\delta u\|_{L^\infty(E_\delta)}
    \|\sigma\|_{L^1(E_\delta,\lambda_\delta)}.
\]

Conversely, assume the stated bound. For \(\varepsilon>0\), put
\[
    E_{\delta,\varepsilon}
    =
    \{(x,x')\in E_\delta:|x-x'|>\varepsilon\}.
\]
For every
\(\sigma\in C_c^\infty(E_\delta)\cap L^1_{\asym}(E_\delta)\), the
explicit divergence formula and Fubini's theorem yield
\[
    \langle u,\operatorname{div}_\delta\sigma\rangle_\O
    =
    -\lim_{\varepsilon\downarrow0}
    \int_{E_{\delta,\varepsilon}}\sigma D_\delta u
    \,\mathrm d\lambda_\delta.
\]
Hence
\(F_u(\sigma)=\langle u,\operatorname{div}_\delta\sigma\rangle_\O\)
defines a bounded linear functional, with respect to the \(L^1\)-norm,
on the smooth antisymmetric fields. By Hahn--Banach and
\((L^1)^*=L^\infty\), there is
\(g\in L^\infty(E_\delta,\lambda_\delta)\) such that
\[
    F_u(\sigma)
    =
    \int_{E_\delta}g\sigma\,\mathrm d\lambda_\delta
\]
for every such \(\sigma\).

Fix \(\varepsilon>0\). Testing with smooth antisymmetric fields supported
in \(E_{\delta,\varepsilon}\) gives
\[
    \int_{E_{\delta,\varepsilon}}(g+D_\delta u)\sigma
    \,\mathrm d\lambda_\delta=0.
\]
Since \(D_\delta u\in L^p(E_{\delta,\varepsilon})\), Lemma~A.1(5),
applied with \(r=q\), shows that \(g+D_\delta u\) is symmetric on
\(E_{\delta,\varepsilon}\). Since \(D_\delta u\) is antisymmetric,
\[
    D_\delta u
    =
    -\frac12(g-g^{\mathrm T})
    \qquad\text{a.e. in }E_{\delta,\varepsilon}.
\]
The union of these truncated sets differs from \(E_\delta\) only by the
diagonal, which has \(\lambda_\delta\)-measure zero. Thus the same identity
holds almost everywhere in \(E_\delta\), and
\(D_\delta u\in L^\infty(E_\delta)\). The first part of the proof now
gives \eqref{eq:closed-L1-integration-by-parts}.
\end{proof}

\begin{proposition}[Dual formulation of the conductivity-uncapped problem]
\label{prop:appendix-uncapped-dual}
For every \(\delta>0\),
\begin{equation}
\label{eq:appendix-uncapped-dual}
    m_\delta
    =
    \sup\left\{
        \langle f,u\rangle_\O:
        u\in L^p(\O),\quad
        \|D_\delta u\|_{L^\infty(E_\delta)}\leq1
    \right\}.
\end{equation}
\end{proposition}

\begin{proof}
On \(Y_{\delta,\asym}^1\), let
\(F(\sigma)=\|\sigma\|_{L^1(E_\delta,\lambda_\delta)}\), and on
\(L^q(\O)\), let \(G=\iota_{\{f\}}\). The operator
\(A=\operatorname{div}_\delta:Y_{\delta,\asym}^1\to L^q(\O)\) is
continuous and surjective. Fenchel--Rockafellar duality therefore gives
\[
    m_\delta
    =
    -\inf_{u\in L^p(\O)}
    \{F^*(A^*u)+G^*(-u)\},
    \qquad
    G^*(-u)=-\langle f,u\rangle_\O.
\]
We claim that
\begin{equation}
\label{eq:appendix-conjugate-identification}
    F^*(A^*u)
    =
    \begin{cases}
        0,
        &\|D_\delta u\|_{L^\infty(E_\delta)}\leq1,\\
        +\infty,
        &\text{otherwise}.
    \end{cases}
\end{equation}
If \(D_\delta u\in L^\infty(E_\delta)\),
Lemma~\ref{lem:bounded-nonlocal-gradient} gives
\[
    F^*(A^*u)
    =
    \sup_{\sigma\in Y_{\delta,\asym}^1}
    \left\{
        -\int_{E_\delta}\sigma D_\delta u
        \,\mathrm d\lambda_\delta
        -\|\sigma\|_{L^1(E_\delta,\lambda_\delta)}
    \right\}.
\]
The smooth antisymmetric fields lie in \(Y_{\delta,\asym}^q\) and are
dense in \(L^1_{\asym}(E_\delta,\lambda_\delta)\). The last supremum is
therefore zero precisely when
\(\|D_\delta u\|_{L^\infty(E_\delta)}\leq1\), and is infinite otherwise.

Conversely, if \(F^*(A^*u)<+\infty\), positive homogeneity gives
\[
    |\langle u,\operatorname{div}_\delta\sigma\rangle_\O|
    \leq
    \|\sigma\|_{L^1(E_\delta,\lambda_\delta)}
    \qquad
    \text{for every }\sigma\in Y_{\delta,\asym}^1.
\]
Lemma~\ref{lem:bounded-nonlocal-gradient} then gives
\(D_\delta u\in L^\infty(E_\delta)\) with norm at most one. This proves
\eqref{eq:appendix-conjugate-identification}, and the dual formula follows.
\end{proof}

\subsection{Exact horizon-independence}
\label{subsec:uncapped-exact-horizon}

\begin{theorem}[Exact horizon-independence]
\label{thm:appendix-exact-uncapped}
For every \(\delta>0\),
\[
    m_\delta=m_0.
\]
\end{theorem}

\begin{proof}
We show that the feasible set in \eqref{eq:appendix-uncapped-dual}
coincides, up to the choice of a representative, with
\[
    \{u\in\operatorname{Lip}_0(\O):\operatorname{Lip}(u)\leq1\}.
\]
If \(u\) belongs to this set, then its zero extension satisfies
\(|u(x)-u(x')|\leq|x-x'|\) for all \(x,x'\in\Rn\), and hence
\(\|D_\delta u\|_{L^\infty(E_\delta)}\leq1\).

Conversely, suppose that
\(\|D_\delta u\|_{L^\infty(E_\delta)}\leq1\), and let \(u_0\) be the zero
extension of \(u\). The change of variables \((x,x')=(x,x+h)\) and
Tonelli's theorem give a null set \(N\subset B_\delta\) such that
\[
    |u_0(x+h)-u_0(x)|\leq|h|
\]
for almost every \(x\in\Rn\) and every \(h\in B_\delta\setminus N\).
Let \(\eta_\varepsilon\) be a standard mollifier and set
\(u_\varepsilon=\eta_\varepsilon*u_0\). Then
\(|u_\varepsilon(x+h)-u_\varepsilon(x)|\leq|h|\) for every
\(x\in\Rn\) and \(h\in B_\delta\setminus N\). Continuity in \(h\)
extends the inequality to all \(h\in B_\delta\). Subdividing an arbitrary
displacement into finitely many increments of length less than \(\delta\)
shows that \(u_\varepsilon\) is globally one-Lipschitz.

Since \(u_\varepsilon\to u_0\) in \(L^p_{\mathrm{loc}}(\Rn)\), the
Arzel\`a--Ascoli theorem yields a locally uniformly convergent subsequence
whose limit is a globally one-Lipschitz representative of \(u_0\). This
representative vanishes on \(\Rn\setminus\O\), and therefore belongs to
\(\operatorname{Lip}_0(\O)\). Proposition~\ref{prop:appendix-uncapped-dual}
now gives \(m_\delta=m_0\).
\end{proof}

\begin{remark}
The identity \(m_\delta=m_0\) is stronger than an asymptotic localization
statement. It is not required for the direct recovery of optimal capped
fluxes in Section~\ref{sec:sparse-flux-recovery}; rather, it provides an
exact reference problem and a nonlocal dual characterization of the global
Lipschitz constraint.
\end{remark}